\documentclass[11pt,reqno]{amsart}
\usepackage[a4paper, hmargin={2.7cm,2.7cm},vmargin={3.3cm,3.3cm}]{geometry}
\usepackage[english]{babel}
\usepackage{color}
\usepackage{amsmath}
\usepackage{amssymb}
\usepackage{amsfonts}
\usepackage{amsthm}
\usepackage{enumerate}
\usepackage{mathrsfs}
\usepackage{amsthm}
\usepackage[noadjust]{cite}
\usepackage{dsfont}
\usepackage{etoolbox}
\usepackage[unicode=true]{hyperref}
\usepackage[noabbrev]{cleveref}
\usepackage{esint}
\usepackage{todonotes}

\binoppenalty=\maxdimen
\relpenalty=\maxdimen
\makeatletter
\patchcmd{\ttlh@hang}{\parindent\z@}{\parindent\z@\leavevmode}{}{}
\patchcmd{\ttlh@hang}{\noindent}{}{}{}
\makeatother

\newcommand\numberthis{\addtocounter{equation}{1}\tag{\theequation}}

\makeatletter
\@namedef{subjclassname@2020}{\textup{2020} Mathematics Subject Classification}
\makeatother

\theoremstyle{plain}
\newtheorem{theorem}{Theorem}[section]
\newtheorem{lemma}[theorem]{Lemma}
\newtheorem{proposition}[theorem]{Proposition}
\newtheorem{corollary}[theorem]{Corollary}

\newtheorem{innercustomgeneric}{\customgenericname}
\providecommand{\customgenericname}{}
\newcommand{\newcustomtheorem}[2]{%
  \newenvironment{#1}[1]
  {%
   \renewcommand\customgenericname{#2}%
   \renewcommand\theinnercustomgeneric{##1}%
   \innercustomgeneric
  }
  {\endinnercustomgeneric}
}

\def\XXint#1#2#3{{\setbox0=\hbox{$#1{#2#3}{\int}$ }
\vcenter{\hbox{$#2#3$ }}\kern-.6\wd0}}

\newcustomtheorem{customthm}{Theorem}
\newcustomtheorem{customlemma}{Lemma}

\theoremstyle{definition}

\theoremstyle{remark}
\newtheorem{remark}[theorem]{Remark}

\numberwithin{equation}{section}

\DeclareMathOperator*{\supp}{supp}

\DeclareMathOperator*{\scale}{scale}

\DeclareMathOperator*{\diam}{diam}
\DeclareMathOperator*{\dist}{dist}

\usepackage{xparse}

\newcommand{\Schwartz}{\mathcal{S}}
\newcommand{\SP}{\mathcal{S}'  / \mathcal{P} }

\newcommand{\TLgeneral}[3]{\dot{\mathbf{F}}_{#1,#2}^{#3}}
\newcommand{\TLzero}{\dot{\mathbf{F}}^{0}_{p,q}}
\newcommand{\TLone}{\dot{\mathbf{F}}^{\alpha}_{p_1,q_1}}
\newcommand{\TLtwo}{\dot{\mathbf{F}}^{\beta}_{p_2,q_2}}
\newcommand{\TLnaked}{\dot{\mathbf{F}}}
\newcommand{\TL}{\dot{\mathbf{F}}^{\alpha}_{p,q}}
\newcommand{\TLi}{\dot{\mathbf{F}}^{\alpha}_{\infty,q}}
\newcommand{\TLii}{\dot{\mathbf{F}}^{\alpha}_{\infty,\infty}}

\NewDocumentCommand\TLseq{O{\alpha}}{\dot{\mathbf{f}}^{#1}_{p,q}}
\NewDocumentCommand\PT{O{\alpha}D<>{\beta}}{\dot{\mathbf{P}}^{#1, #2}_{p,q}}
\NewDocumentCommand\PTi{O{\alpha}D<>{\beta}}{\dot{\mathbf{P}}^{#1, #2}_{\infty,q}}
\NewDocumentCommand\PTii{O{\alpha}D<>{\beta}}{\dot{\mathbf{P}}^{#1, #2}_{\infty,\infty}}

\newcommand{\TLAii}{\dot{\mathbf{F}}^{\alpha}_{\infty, \infty}(A)}
\newcommand{\BAii}{\dot{\mathbf{B}}^{\alpha}_{\infty, \infty}(A)}

\NewDocumentCommand\PTseq{O{\alpha}D<>{\beta}}{\dot{\mathbf{p}}^{#1, #2}_{\infty,q}}

\newcommand{\vertiii}[1]{{\left\vert \kern-0.25ex
                            \left\vert \kern-0.25ex
                              \left\vert #1\right\vert\kern-0.25ex
                            \right\vert \kern-0.25ex
                          \right\vert}}

\NewDocumentCommand\DoubleStar{O{\varphi}m}{#1_{#2,\beta}^{\ast\ast}}

\renewcommand{\emptyset}{\varnothing}

\newcommand{\CalP}{\mathcal{P}}
\newcommand{\CalD}{\mathcal{D}}
\newcommand{\Measure}{\mathrm{m}}
\newcommand{\Lebesgue}[1]{\Measure \left( #1 \right)}
\newcommand{\GL}{\operatorname{GL}}
\newcommand{\Fourier}{\mathcal{F}}

\DeclareFontFamily{U}{mathx}{\hyphenchar\font45}
\DeclareFontShape{U}{mathx}{m}{n}{
      <5> <6> <7> <8> <9> <10>
      <10.95> <12> <14.4> <17.28> <20.74> <24.88>
      mathx10
      }{}
\DeclareSymbolFont{mathx}{U}{mathx}{m}{n}
\DeclareFontSubstitution{U}{mathx}{m}{n}
\DeclareMathAccent{\widecheck}{0}{mathx}{"71}
\DeclareMathAccent{\wideparen}{0}{mathx}{"75}

\newcommand{\R}{\mathbb{R}}

\newcommand{\CC}{\mathbb{C}}
\newcommand{\N}{\mathbb{N}}
\newcommand{\Z}{\mathbb{Z}}
\newcommand{\EE}{\mathbb{E}}
\newcommand{\PP}{\mathbb{P}}

\title{Classification of anisotropic Triebel-Lizorkin spaces}

\author[S. Koppensteiner]{Sarah Koppensteiner}
\address{Faculty of Mathematics,
University of Vienna,
Oskar-Morgenstern-Platz 1,
A-1090 Vienna, Austria}
\email{sarah.koppensteiner@univie.ac.at}

\author[J.T. van Velthoven]{Jordy Timo van Velthoven}
\address{Delft University of Technology,
Mekelweg 4, Building 36,
2628 CD Delft, The Netherlands.}
\email{j.t.vanvelthoven@tudelft.nl}

\author[F. Voigtlaender]{Felix Voigtlaender}

\address{
Mathematical Institute for Machine Learning and Data Science (MIDS),
Catholic University of Eichstätt-Ingolstadt (KU),
Auf der Schanz 49,
85049 Ingolstadt,
Germany.
}
\email{felix.voigtlaender@ku.de}



\begin{document}

\begin{abstract}
This paper provides a classification theorem for expansive matrices $A \in \mathrm{GL}(d, \mathbb{R})$
generating the same anisotropic homogeneous Triebel-Lizorkin space $\TL(A)$ for $\alpha \in \mathbb{R}$
and $p,q \in (0,\infty]$.
It is shown that $\TL(A) = \TL(B)$ if and only if the homogeneous quasi-norms $\rho_A, \rho_B$
associated to the matrices $A, B$ are equivalent,
except for the case $\dot{\mathbf{F}}^0_{p, 2} = L^p$ with $p \in (1,\infty)$.
The obtained results complement and extend the classification of anisotropic Hardy spaces
$H^p(A) = \dot{\mathbf{F}}^{0}_{p,2}(A)$, $p \in (0,1]$, in [Mem. Am. Math. Soc. 781, 122 p. (2003)].
\end{abstract}

\maketitle

\section{Introduction}

Let $A \in \mathrm{GL}(d, \mathbb{R})$ be an expansive matrix and consider an analyzing vector
$\varphi \in \mathcal{S}(\mathbb{R}^d)$ for $A$, that is, a Schwartz function
$\varphi : \mathbb{R}^d \to \mathbb{C}$ with Fourier transform
$\widehat{\varphi} \in C_c^{\infty} (\mathbb{R}^d \setminus \{0\})$ satisfying
\begin{align*}
 \sup_{i \in \mathbb{Z}} |\widehat{\varphi} ((A^*)^i \xi )| > 0
 \quad \text{for all} \quad \xi \in \mathbb{R}^d \setminus \{0\},
\end{align*}
where $A^*$ denotes the transpose of $A$.
Denote its $L^1$-normalized dilation by $\varphi_i \!:=\! |\det A|^i \varphi (A^i \cdot)$
for $i \in \mathbb{Z}$.
For $\alpha \in \mathbb{R}$ and $p, q \in (0,\infty]$,
the associated anisotropic homogeneous Triebel-Lizorkin space $\TL(A)$ on $\mathbb{R}^d$
is defined to consist of all tempered distributions $f \in \mathcal{S}' (\mathbb{R}^d)$
(modulo polynomials) with finite quasi-norm $\| f \|_{\TL(A)}$, defined by
\[
  \| f \|_{\TL(A)}
  = \bigg\|
      \bigg(
        \sum_{i \in \mathbb{Z}}
          (|\det A|^{\alpha i} \, |f \ast \varphi_i |)^q
      \bigg)^{1/q}
    \bigg\|_{L^p},
  \quad p \in (0,\infty),
\]
with the usual modifications for $q = \infty$, and
\[
 \| f \|_{\TLi(A)}
  = \sup_{\ell \in \mathbb{Z}, k \in \mathbb{Z}^d}
       \bigg(
         \frac{1}{|\det A|^{\ell}}
         \int_{A^{\ell}([0,1]^d + k)}
           \sum_{i = -\ell}^{\infty}
             (|\det A|^{\alpha i} \, |(f \ast \varphi_i)(x)|)^q \;
         d x
       \bigg)^{1/q},
\]
and $\| f \|_{\TLii(A)} = \sup_{i \in \mathbb{Z}} |\det A|^{\alpha i} \| f \ast \varphi_i \|_{L^{\infty}}$.

For the scalar dilation matrix $A = 2 \cdot I_d$, the spaces $\TL(A)$ defined above
coincide with the usual homogeneous Triebel-Lizorkin spaces on $\mathbb{R}^d$ as studied in,
e.g., \cite{triebel2010theory, frazier1991littlewood, frazier1990discrete}.
For this particular case, the Triebel-Lizorkin spaces provide a unifying scale
of function spaces that encompasses, among others, the Lebesgue, Sobolev, Hardy and BMO spaces.
The anisotropic Triebel-Lizorkin spaces $\TL(A)$ associated to a general expansive matrix $A$
were first introduced in \cite{bownik2006atomic} and further studied in,
e.g., \cite{bownik2008duality, bownik2007anisotropic, koppensteiner2022anisotropic2, KvVV2021anisotropic, cabrelli2013non, li2014mean, benyi2010anisotropic}.
These anisotropic spaces are useful for the analysis of mixed homogeneity properties of functions
and operators as the dilation structure allows different directions to be scaled
by different dilation factors.
Among others, the anisotropic Triebel-Lizorkin spaces include Lebesgue spaces
and various anisotropic/parabolic versions of Hardy and BMO spaces as studied in,
e.g., \cite{bownik2003anisotropic, folland1982hardy, calderon1977parabolic, calderon1975parabolic, calderon1977atomic, bownik2020pde}.
See these papers (and the references therein) for further motivation
for considering anisotropic function spaces.

In the present paper, the main objective is to characterize when two expansive matrices
induce the same anisotropic Triebel-Lizorkin space.
For the special case of anisotropic Hardy spaces $H^p(A)$ ($= \dot{\mathbf{F}}^0_{p,2} (A)$) with $p \in (0,1]$,
a full solution to this problem has been obtained in \cite{bownik2003anisotropic}.
Explicitly, it is shown in \cite[Section~10]{bownik2003anisotropic} that $H^p(A) = H^p(B)$
for some (equivalently, all) $p \in (0,1]$ if and only if two homogeneous quasi-norms
$\rho_A, \rho_B : \mathbb{R}^d \to [0,\infty)$ associated to the expansive matrices $A, B$
are equivalent, in the usual sense of quasi-norms.
See also \cite{bownik2020pde} for a slightly corrected version and
\cite{dekel2011hardy} for an extension of the classification result
of \cite{bownik2003anisotropic} to Hardy spaces with variable anisotropy.
Analogous to these results on Hardy spaces, a classification of anisotropic Besov spaces
\cite{bownik2005atomic} has more recently been obtained in \cite{cheshmavar2020classification}.
The aim of this paper is to provide a complementary characterization
for the scale of Triebel-Lizorkin spaces.

\subsection{Main results}

The first key result obtained in this paper is the following rigidity theorem.
Here, as well as below, two expansive matrices $A$ and $B$ are called \emph{equivalent}
if they have equivalent homogeneous quasi-norms;
see \Cref{sec:expansive,sec:QuasiNorms} for precise definitions.

\begin{theorem}\label{thm:intro}
  Let $A, B \in \mathrm{GL}(d, \mathbb{R})$ be expansive,
  $\alpha, \beta \in \mathbb{R}$ and $p_1, p_2, q_1, q_2 \in (0, \infty]$.

  If $\TLone(A) = \TLtwo(B)$, then $(p, q, \alpha) := (p_1, q_1, \alpha) = (p_2, q_2, \beta)$.
  Furthermore, at least one of the following conditions hold:
  \begin{enumerate}[(i)]
   \item $A$ and $B$ are equivalent, or
   \item  $\alpha = 0$, $p \in (1, \infty)$ and $q = 2$.
  \end{enumerate}
\end{theorem}

\Cref{thm:intro} shows, in particular, that equivalence of two expansive matrices is necessary
for the coincidence of the associated spaces, unless $\alpha = 0$, $p \in (1, \infty)$ and $q = 2$.
That this conclusion might fail for the space $\dot{\mathbf{F}}^0_{p, 2} (A)$ with $p \in (1,\infty)$
is easily explained, namely $\dot{\mathbf{F}}^0_{p, 2} (A)$ can be canonically identified
with the Lebesgue space $L^p$ for $p\in(1,\infty)$, see, e.g., \cite{bownik2007anisotropic, bownik2003anisotropic}.

The following theorem provides a converse to \Cref{thm:intro}.

\begin{theorem} \label{thm:intro2}
  If $A, B \in \mathrm{GL}(d, \mathbb{R})$ are equivalent expansive matrices,
  then $\TL(A) = \TL(B)$ for all $\alpha \in \mathbb{R}$ and $p,q \in (0,\infty]$.
\end{theorem}

A combination of \Cref{thm:intro} and \Cref{thm:intro2} provides a full characterization
of two expansive matrices inducing the same anisotropic Triebel-Lizorkin space.
This characterization extends the classification of anisotropic Hardy spaces \cite{bownik2003anisotropic}
to the full scale of Triebel-Lizorkin spaces, while complementing the classification
of anisotropic Besov spaces \cite{cheshmavar2020classification}
with a counterpart for Triebel-Lizorkin spaces.

In effect, the aforementioned classification theorems  translate the problem
of comparing function spaces into the comparison of homogeneous quasi-norms.
For this latter problem, explicit and verifiable criteria in terms of spectral properties
of the involved dilation matrices can be given, see, e.g.,
\cite[Section~10]{bownik2003anisotropic}, \cite[Section~7]{cheshmavar2020classification}
and \cite[Section~4]{bownik2020pde}.

As an illustration of \Cref{thm:intro}, we note that a matrix $B \in \mathrm{GL}(d, \mathbb{R})$
is equivalent to the scalar dilation $A = 2 \cdot I_d$ if and only if $B$ is diagonalizable
over $\mathbb{C}$ with all eigenvalues equal in absolute value, see, e.g.,
\cite[Example, p.7]{bownik2003anisotropic}.
Combined with \Cref{thm:intro}, this shows that for matrices $B$ that are not of this special form,
\[
  \TL(A) \neq \TL(B),
\]
unless $\alpha = 0$, $p \in (1,\infty)$ and $q = 2$.
In particular, the (homogeneous) Sobolev spaces $L^p_{\alpha}$ ($= \dot{\mathbf{F}}^{\alpha}_{p, 2} (A)$)
with $1 < p < \infty$ and $\alpha  \neq 0$ do \emph{not} coincide
with $\dot{\mathbf{F}}^{\alpha}_{p, 2} (B)$ for non-diagonalizible matrices $B$.

Lastly, let us mention an application of \Cref{thm:intro2}.
In \cite{koppensteiner2022anisotropic2, KvVV2021anisotropic},
we proved \emph{continuous} maximal characterizations of anisotropic Triebel-Lizorkin spaces $\TL(A)$
and obtained new results on their molecular decomposition.
These results were obtained under the additional assumption that the expansive matrix $A$ is exponential,
in the sense that $A = \exp(C)$ for some matrix $C \in \mathbb{R}^{d \times d}$.
\Cref{thm:intro2} implies that this additional assumption does not restrict
the scale of anisotropic Triebel-Lizorkin spaces.
Indeed, since there always exists an expansive \emph{and} exponential matrix $B$
that is equivalent to the given expansive matrix $A$ (cf.\ \mbox{\cite[Section~7]{cheshmavar2020classification}}),
it follows by \Cref{thm:intro2} that $\TL(A) = \TL(B)$ for all $\alpha \in \mathbb{R}$ and $p,q \in (0,\infty]$.

\subsection{Methods}

An essential ingredient in our proof of \Cref{thm:intro} and \Cref{thm:intro2}
is a simple characterization of the equivalence of two expansive matrices $A$ and $B$
in terms of properties of the associated covers $\bigl( (A^*)^i Q\bigr)_{i \in \mathbb{Z}}$
and $\bigl( (B^*)^j P\bigr)_{j \in \mathbb{Z}}$ of $ \mathbb{R}^d \setminus \{0\}$,
where $P, Q \subseteq \mathbb{R}^d \setminus \{0\}$ are suitable relatively compact sets;
see \cite[Lemma~6.2]{cheshmavar2020classification} and \Cref{sec:covers}.
Explicitly, this criterion asserts that two expansive matrices $A, B$ are equivalent
if and only if the associated homogeneous covers $\bigl( (A^*)^i Q\bigr)_{i \in \mathbb{Z}}$
and $\bigl( (B^*)^j P\bigr)_{j \in \mathbb{Z}}$ satisfy
\begin{align} \label{eq:equiv_matrix_intro}
 \sup_{i \in \mathbb{Z}}
   \big| \big\{ j \in \mathbb{Z} : (A^*)^i Q \cap (B^*)^j P \neq \emptyset \big\} \big|
 + \sup_{j \in \mathbb{Z}}
     \big| \big\{ i \in \mathbb{Z} : (A^*)^i Q \cap (B^*)^jP \neq \emptyset \big\} \big|
  < \infty.
\end{align}
The formulation \eqref{eq:equiv_matrix_intro} of the equivalence of matrices $A$ and $B$
is what is actually used in the proofs of our main results, as we expand upon next.

\subsubsection*{Necessary conditions.}

In the proof of \Cref{thm:intro}, we show the asserted equivalence of two matrices $A$ and $B$
by showing that the criterion \eqref{eq:equiv_matrix_intro} holds.
For this, we first carefully construct auxiliary functions in $\TL(A) = \TL(B)$ whose Fourier supports
are contained in finitely many of the sets $(A^*)^{i_k} Q$ and $(B^*)^{j_k} P$,
where $i_k, j_k \in \mathbb{Z}$, of appropriate homogeneous covers $\bigl( (A^*)^i Q\bigr)_{i \in \mathbb{Z}}$
and $\bigl( (B^*)^j P\bigr)_{j \in \mathbb{Z}}$.
Then, using adequate estimates of the norms of these auxiliary functions
(see \Cref{sec:norm_estimates}), it is shown directly that \eqref{eq:equiv_matrix_intro}
must hold for the case $\alpha \neq 0$, in which case $A$ and $B$ must be equivalent.
The proof strategy for  the case $\alpha = 0$ is similar,
but requires some additional arguments and tools.
For $p < \infty$, it is shown using the Khintchine inequality that necessarily $q = 2$
whenever $A$ and $B$ are not equivalent.
For $p = \infty$, we use dual norm characterizations of Triebel-Lizorkin norms
to conclude that $A$ and $B$ must be equivalent.

\subsubsection*{Sufficient conditions.}

In the proof of \Cref{thm:intro2}, the criterion \eqref{eq:equiv_matrix_intro} is used
to control the overlap of the Fourier supports of the $A$-dilates and $B$-dilates
of the analyzing vectors $\varphi$ and $\psi$, respectively,
that are used to define the spaces $\TL(A)$ and $\TL(B)$.
Combined with our maximal characterizations of Triebel-Lizorkin spaces
obtained in \cite{koppensteiner2022anisotropic2, KvVV2021anisotropic},
this allows to conclude that the analyzing vectors $\varphi$ and $\psi$ for $A$ respectively $B$
define the same space $\TL(A) = \TL(B)$.
\\~\\ \indent
As mentioned above, the used criterion \eqref{eq:equiv_matrix_intro} for equivalent matrices stems
from \cite{cheshmavar2020classification}, where it was used for the purpose
of classifying anisotropic Besov spaces.
For the actual comparison of function spaces, the approach of \cite{cheshmavar2020classification}
consists of showing that an anisotropic Besov space can be identified
with a (Besov-type) decomposition space \cite{VoigtlaenderEmbeddingsOfDecompositionSpaces},
which allows to apply the embedding theory \cite{VoigtlaenderEmbeddingsOfDecompositionSpaces}
developed by the third named author.
In contrast, the Triebel-Lizorkin spaces considered in this paper cannot be directly treated in the framework \cite{VoigtlaenderEmbeddingsOfDecompositionSpaces}; in particular, our main theorems cannot be easily deduced from \cite{VoigtlaenderEmbeddingsOfDecompositionSpaces}.
Some of our arguments for proving \Cref{thm:intro} are, however, inspired by ideas
used in \cite{VoigtlaenderEmbeddingsOfDecompositionSpaces},
most notably the use of the Khintchine inequality.
Nevertheless, all of our calculations and estimates differ non-trivially
from corresponding arguments in \cite{VoigtlaenderEmbeddingsOfDecompositionSpaces}
as the latter concerns Besov-type norms, which are technically easier to deal with
than the Triebel-Lizorkin norms considered in this paper.

\subsection{Organization}

The overall structure of this paper is as follows:
\Cref{sec:expansive0} collects various notions and results related to expansive matrices
and associated homogeneous covers.
The essential background on anisotropic Triebel-Lizorkin spaces is contained in \Cref{sec:TL0}.
\Cref{thm:intro} is proven in \Cref{sec:necessary}, whereas \Cref{sec:sufficient} provides the proof
of \Cref{thm:intro2}.
Lastly, some technical auxiliary results  are postponed to two appendices.

\subsection{Notation}

For a measurable set $\Omega \subseteq \mathbb{R}^d$,
we denote its Lebesgue measure by $\Measure(\Omega)$ and the indicator function of $\Omega$
by $\mathds{1}_{\Omega}$.
The notation $| \cdot | : \mathbb{R}^d \to [0, \infty)$ is used for the Euclidean norm.
The open Euclidean ball of radius $r > 0$ and center $x \in \mathbb{R}^d$ is denoted by $B_r (x)$.
The closure of a set $\Omega \subseteq \R^d$ will be denoted by $\overline{\Omega}$.

The Schwartz space on $\mathbb{R}^d$ is denoted by $\mathcal{S} (\mathbb{R}^d)$
and $\mathcal{S}' (\mathbb{R}^d)$ denotes its dual, the space of tempered distributions.
For $f \in \Schwartz'(\R^d)$ and $g \in \Schwartz(\R^d)$,
we define $\langle f, g \rangle := f(\overline{g})$,
so that the dual pairing $\langle \cdot, \cdot \rangle$ is sesquilinear,
in agreement with the inner product on $L^2(\R^d)$.
The subspace of $\mathcal{S}(\mathbb{R}^d)$ consisting of functions
with all moments vanishing (i.e., $\int x^\alpha \, f(x) \, dx = 0$ for all $\alpha \in \N_0^d$)
is denoted by $\mathcal{S}_0 (\mathbb{R}^d)$.
The dual space $\mathcal{S}'_0 (\mathbb{R}^d)$ is often identified
with the quotient $\SP$ of $\mathcal{S}'(\mathbb{R}^d)$ and the space of polynomials $\mathcal{P}(\mathbb{R}^d)$.
Lastly, the space of smooth compactly supported functions on an open set $U \subseteq \mathbb{R}^d$
is as usual denoted by $C^{\infty}_c (U)$.

For a function $f : \mathbb{R}^d \to \mathbb{C}$, its translation $T_y f$ and modulation $M_y f$
by $y \in \mathbb{R}^d$ are defined by $T_y f = f(\cdot - y)$ and $M_y f = e^{2\pi i y \cdot} f$,
respectively.
The Fourier transform of $f \in L^1 (\mathbb{R}^d)$ is normalized as
$\widehat{f} (\xi) = \int_{\mathbb{R}^d} f(x) e^{- 2 \pi i \xi \cdot x} \; dx$
for $\xi \in \mathbb{R}^d$, where $\xi \cdot x = \sum_{j=1}^d \xi_j x_j$.
The notation $\mathcal{F} f := \widehat{f}$ is also sometimes used.

For two functions $f, g : X \to [0,\infty)$ on a set $X$, we write $f \lesssim g$
whenever there exists $C > 0$ such that $f(x) \leq C g (x)$ for all $x \in X$.
We simply use the notation $f \asymp g$ whenever $f \lesssim g$ and $g \lesssim f$.
We also write $A \lesssim B$ for the inequality $A \leq C B$,
where $C > 0$ is constant independent of $A$ and $B$.
In case the implicit constant in $\lesssim$ depends on a quantity $\alpha$,
we also sometimes write $\lesssim_\alpha$.

\section{Expansive matrices and homogeneous covers}
\label{sec:expansive0}

This section collects background on expansive matrices and homogeneous quasi-norms.
A standard reference for most of the presented material is \cite{bownik2003anisotropic}.

\subsection{Expansive matrices}
\label{sec:expansive}

A matrix $A \in \mathrm{GL}(d, \R)$ is said to be \emph{expansive}
if $|\lambda| > 1$ for all $\lambda \in \sigma(A)$,
where $\sigma(A) \subseteq \CC$ denotes the spectrum of $A$.
Let $\lambda_-$ and $\lambda_+$ denote numbers
such that $1 < \lambda_- < \min_{\lambda \in \sigma(A)} |\lambda|$
and $\lambda_+ > \max_{\lambda \in \sigma(A)} |\lambda|$,
and define $\zeta_+ := \ln \lambda_+ / \ln |\det A|$ and $ \zeta_- := \ln \lambda_- / \ln |\det A|$.
Then there exists $C \geq 1$ such that, for all $x \in \mathbb{R}^d$,
\begin{equation}
  \begin{split}
    \frac{1}{C}  [\rho_A (x)]^{\zeta_-} &\leq | x | \leq C  [\rho_A(x)]^{\zeta_+},
    \quad \text{if }\rho_A(x) \geq 1, \\
    \frac{1}{C}  [\rho_A (x)]^{\zeta_+} &\leq | x | \leq C  [\rho_A(x)]^{\zeta_-},
    \quad \text{if }\rho_A(x) \leq 1,
  \end{split}
  \label{eq:ExpansiveConsequence}
\end{equation}
see, e.g., \cite[Lemma~3.2]{bownik2003anisotropic}.

A set $\Omega \subseteq \R^d$ is an \emph{ellipsoid} if $\Omega = \{ x \in \mathbb{R}^d : |P x| < 1 \}$
for some $P \in \mathrm{GL}(d, \mathbb{R})$.
Given any expansive matrix $A$, there exists an ellipsoid $\Omega_A$ and $r > 1$ such that
\begin{align} \label{eq:expansive_ellipsoid}
  \Omega_A \subseteq r \Omega_A \subseteq A \Omega_A,
\end{align}
and $\Lebesgue{\Omega_A} = 1$, see, e.g., \cite[Lemma 2.2]{bownik2003anisotropic}.
The choice of an ellipsoid satisfying \eqref{eq:expansive_ellipsoid} is not unique.
Throughout, given an expansive matrix $A$, we will fix one choice of ellipsoid $\Omega_A$ associated to $A$.

\subsection{Homogeneous quasi-norms}
\label{sec:QuasiNorms}

Let $A \in \mathrm{GL}(d, \mathbb{R})$ be an expansive matrix.
A \emph{homogeneous quasi-norm} associated with $A$ is a measurable function
$\rho : \mathbb{R}^d \to [0,\infty)$ satisfying the three properties:
\begin{enumerate}
  \item[(q1)] $\rho (x) = 0$ if and only if $x = 0$;
  \item[(q2)] $\rho(A x) = |\det A| \rho(x)$ for all $x \in \mathbb{R}^d$;
  \item[(q3)] there exists $C > 0$ such that $\rho(x+y) \leq C(\rho(x) + \rho(y))$
              for all $x, y \in \mathbb{R}^d$.
\end{enumerate}
By \cite[Lemma 2.4]{bownik2003anisotropic}, any two homogeneous quasi-norms $\rho_1, \rho_2$
associated to a fixed expansive matrix $A$ are equivalent,
in the sense that there exists $C > 0$ such that
\begin{align} \label{eq:equiv_homnorms}
  \frac{1}{C} \rho_1 (x) \leq \rho_2 (x) \leq C \rho_1 (x)
\end{align}
for all $x \in \mathbb{R}^d$.

In the sequel, we will primarily work with the so-called \emph{step homogeneous quasi-norm}
$\rho_A$ associated to $A$, defined as
\[
  \rho_A (x)
  = \begin{cases}
      |\det A|^i, & \text{if} \quad x \in A^{i+1} \Omega_A \setminus A^i \Omega_A, \\
      0,          & \text{if} \quad x = 0,
    \end{cases}
\]
where $\Omega_A$ is the fixed expansive ellipsoid \eqref{eq:expansive_ellipsoid};
see \cite[Definition~2.5]{bownik2003anisotropic}.

Two expansive matrices $A, B \in \mathrm{GL}(d, \mathbb{R})$ are called \emph{equivalent}
if the associated step homogeneous quasi-norms $\rho_{A}$ and $\rho_{B}$ are equivalent.
Note that, by \Cref{eq:equiv_homnorms}, two expansive matrices are equivalent if and only if
all of their associated quasi-norms are equivalent.

The following characterization is \cite[Lemma 10.2]{bownik2003anisotropic}.

\begin{lemma}[\cite{bownik2003anisotropic}]
  Let $A,B \in \GL(d, \R)$ be expansive.
  Then $A$ and $B$ are equivalent if and only if
  \[
    \sup_{k \in \mathbb{Z}}
      \big\| A^{-k} B^{\lfloor c k \rfloor} \big\|
    < \infty,
  \]
  where $c = c(A, B) := \ln |\det A| / \ln |\det B|$.
\end{lemma}

As a corollary of the previous lemma (see also \cite[Remark~4.9]{cheshmavar2020classification}),
we see that equivalence of expansive matrices is preserved under taking transposes.

\begin{corollary}\label{cor:adjoint_equivalent}
  Two expansive matrices $A$ and $B$ are equivalent
  if and only if $A^*$ and $B^*$ are equivalent.
\end{corollary}

\subsection{Homogeneous covers}
\label{sec:covers}

Let $A \in \mathrm{GL}(d, \mathbb{R})$ be expansive and let
$Q \subseteq \mathbb{R}^d$ be open such that $\overline{Q}$ is compact
in $\mathbb{R}^d \setminus \{0\}$.  A cover
$(A^i Q)_{i \in \mathbb{Z}}$ of $\mathbb{R}^d \setminus \{0\}$ is
called a \emph{homogeneous cover induced by $A$}.  Given two
homogeneous covers $(A^i Q)_{i \in \mathbb{Z}}$ and
$(B^j P)_{j \in \mathbb{Z}}$ induced by
$A, B \in \mathrm{GL}(d, \mathbb{R})$, we define
\begin{align} \label{eq:Ji-Ij}
    J_i
    := \big\{ k \in \mathbb{Z} :  A^i Q \cap B^k P \neq \emptyset \big\}
    \quad \text{and} \quad
    I_j := \big\{ k \in \mathbb{Z} : A^k Q \cap B^j P \neq \emptyset \big\}
  \end{align}
for fixed $i, j \in \mathbb{Z}$.

The index sets defined in \Cref{eq:Ji-Ij} can be used for
characterizing the equivalence of two expansive matrices as the
following lemma shows. See \cite[Lemma
6.2]{cheshmavar2020classification} for a proof.

\begin{lemma}[\cite{cheshmavar2020classification}] \label{lem:hom_covers}
  Let $A, B \in \mathrm{GL}(d, \mathbb{R})$ be expansive
  and let $(A^i Q)_{i \in \mathbb{Z}}$ and $(B^jP)_{j \in \mathbb{Z}}$
  be associated induced covers of $\mathbb{R}^d \setminus \{0\}$.
  Then the step homogeneous quasi-norms $\rho_A$ and $\rho_B$ are equivalent if and only if
  \[
    \sup_{i \in \mathbb{Z}} |J_i | + \sup_{j \in \mathbb{Z}} |I_j| < \infty.
  \]
\end{lemma}

In addition to \Cref{lem:hom_covers}, we will also make use of more
refined estimates on the cardinalities of the index sets defined in
\Cref{eq:Ji-Ij}. We provide the required estimates in the following two lemmata.
The provided proofs follow arguments in the proof of Lemma~\ref{lem:hom_covers}
(cf.\ \cite[Lemma~6.2]{cheshmavar2020classification}) closely, but are included
here for completeness.

\begin{lemma}\label{lem:detQuotient}
  Let $A, B \in \mathrm{GL}(d, \mathbb{R})$ be two equivalent expansive matrices and
  $Q, P \subseteq \R^d$ open such that $\overline{Q}, \overline{P}$ are
  compact in $\mathbb{R}^d \setminus \{0\}$.
  Then there exists $C > 0$ such that
  \begin{equation}
    \label{eq:detQuotient}
    \frac{1}{C} |\det B|^j
    \leq |\det A|^i
    \leq C |\det B|^j
  \end{equation}
  whenever $i,j \in \Z$ are such that $ A^i Q \cap B^j P \neq \emptyset$.
\end{lemma}

\begin{proof}
  If $A^i Q \cap B^j P \neq \emptyset$, then there exists
  $x_0 \in Q \cap A^{-i}B^j P$.  Hence, by homogeneity of
  $\rho_A, \rho_B$ and the assumption of their equivalence, it follows that
  \begin{equation*}
    |\det B|^j  \rho_B(B^{-j} A^i x_0) = \rho_B(A^i x_0) \geq \frac{1}{C}  \rho_A(A^i x_0)
    = \frac{|\det A|^i}{C}  \rho_A(x_0).
  \end{equation*}
  Since $B^{-j} A^i x_0 \in P$, this yields
  \begin{equation*}
    |\det A|^i
    \leq C
    \frac{\max_{x \in \overline{P}}\{\rho_B(x)\}}{\min_{x \in \overline{Q}}\{\rho_A(x)\}} \, |\det B|^j,
  \end{equation*}
  where
  $\max_{x \in \overline{P}}\{\rho_B(x)\} / \min_{x \in
    \overline{Q}}\{\rho_A(x)\}$ is finite by
  \Cref{eq:ExpansiveConsequence} as
  $\overline{Q}, \overline{P}$ are compact in
  $\mathbb{R}^d \setminus \{0\}$.
  The left inequality of
  \eqref{eq:detQuotient} follows analogously by using that
  \begin{equation*}
    |\det B|^j  \rho_B(B^{-j} A^i x_0)  \leq C  \rho_A(A^i x_0)
    = C  |\det A|^i \rho_A(x_0),
  \end{equation*}
  which completes the proof.
\end{proof}

We also need the following estimates involving parameters $\alpha, \beta \in \mathbb{R}$.

\begin{lemma}
   \label{lem:IjJiInclusions-with-alpha-beta}
  Let $A, B \in \mathrm{GL}(d, \mathbb{R})$ be expansive, let $\alpha, \beta \in \R \setminus \{0\}$,
  and let $Q, P \subseteq \R^d$ be open such that $\overline{Q}, \overline{P}$ are
  compact in $\mathbb{R}^d \setminus \{0\}$.
  If there exists $C > 0$ such that
  \begin{equation}
    \label{eq:WeightEquivAssumption}
    \frac{1}{C} |\det B|^{\beta j}
    \leq |\det A|^{\alpha i}
    \leq C |\det B|^{\beta j} \qquad \text{whenever } A^i Q \cap B^j P \neq \emptyset,
  \end{equation}
  then there exists $N \in \N$ such that, for all $i, j \in \Z$,
   \begin{equation*}
      J_i \subseteq \Big\{ j \in \Z \colon \Big|j - \Big\lfloor
      \frac{\alpha}{\beta} \, c \, i \Big\rfloor\Big| \leq N \Big\}
      \qquad \text{and} \qquad
      I_j \subseteq \Big\{ i \in \Z \colon \Big|i - \Big\lfloor
      \frac{\beta}{\alpha} \, \frac{1}{c} \, j \Big\rfloor\Big| \leq N \Big\},
  \end{equation*}
  where  $c = c(A, B) := \ln |\det A| / \ln |\det B|$.
\end{lemma}

\begin{proof}
  Taking the logarithm of \Cref{eq:WeightEquivAssumption} yields
  \begin{equation*}
    \beta j \, \ln (|\det B|) - \ln(C)
    \leq \alpha i  \, \ln(|\det A|)
    \leq \beta j \, \ln(|\det B|) + \ln(C),
  \end{equation*}
  and thus $\big| \alpha i \ln (|\det A|) - \beta j \ln(|\det B|) \big| \leq \ln (C)$.
  This easily implies that
  \begin{equation*}
    \Big|
      i - j \frac{\beta}{\alpha} \frac{\ln(|\det B|)}{\ln(|\det A|)}
    \Big|
    \leq \frac{\ln(C)}{|\alpha| \, \ln (|\det A|)}
    .
  \end{equation*}
  Setting
  $ N_1 := \big\lceil \frac{\ln(C)}{|\alpha| \, \ln(|\det A|)}
  \big\rceil + 1 $, it follows that
  \begin{equation*}
    I_j
    \subseteq \Big\{ i \in \Z \colon \Big|i - \Big\lfloor \frac{\beta}{\alpha} \, \frac{1}{c} \, j
    \Big\rfloor\Big|  \leq N_1 \Big\} .
  \end{equation*}
  The desired inclusion for $J_i$ is obtained analogously with
  $ N_2 := \big\lceil \frac{\ln(C)}{|\beta| \, \ln(|\det B|)}
  \big\rceil + 1 $, which completes the proof by setting $N:= \max \{N_1, N_2\}$.
\end{proof}

\begin{corollary}
  \label{lem:IjJiInclusions}
  Let $A, B \in \mathrm{GL}(d, \mathbb{R})$ be equivalent
  expansive matrices and $Q, P \subseteq \R^d$ open such that
  $\overline{Q}, \overline{P}$ are compact in
  $\mathbb{R}^d \setminus \{0\}$.  Then
  there exists $N \in \N$ such that, for all $i, j \in \Z$,
   \begin{equation*}
     J_i \subseteq \{ j \in \Z \colon |j - \lfloor
     c \, i \rfloor| \leq N \}
     \qquad \text{and} \qquad
     I_j \subseteq \{ i \in \Z \colon |i - \lfloor
     j/c \rfloor| \leq N \},
   \end{equation*}
  where  $c = c(A, B) := \ln |\det A| / \ln |\det B|$.
\end{corollary}

\begin{proof}
  This follows from Lemma~\ref{lem:detQuotient} and
  Lemma~\ref{lem:IjJiInclusions-with-alpha-beta} with
  $\alpha = \beta = 1$.
\end{proof}

Lastly, for a single homogeneous cover $(A^i Q)_{i \in \mathbb{Z}}$, we also define the index set
\[
  N_i (A)
  := \big\{
       k \in \Z
       \colon
       A^i Q \cap A^k Q \neq \emptyset
     \big\}
  .
\]
Note that $N_i (A)$ coincides with the index sets in \eqref{eq:Ji-Ij}
for the choice $A=B$ and $Q = P$.
Therefore, the following is a direct consequence of \Cref{lem:IjJiInclusions}.

\begin{corollary}\label{lem:NeighborSetControl}
  Let $A \in \mathrm{GL}(d, \mathbb{R})$ be expansive and $Q \subseteq \R^d$ open such that
  $\overline{Q}$ is compact in $\mathbb{R}^d \setminus \{0\}$.
  Then there exists  $N \in \N$ such that, for all $i \in \Z$,
  \begin{equation*}
    N_i (A)
    \subseteq \{ j \in \Z \colon |j - i| \leq N \}.
  \end{equation*}
\end{corollary}

\section{Anisotropic Triebel-Lizorkin spaces}
\label{sec:TL0}

Throughout this section, let $A \in \mathrm{GL}(d, \mathbb{R})$ be expansive
and $\Omega_A$ be an associated ellipsoid.

\subsection{Analyzing vectors}

A vector $\varphi \in \mathcal{S}(\mathbb{R}^d)$ is called an \emph{$A$-analyzing vector}
if its Fourier transform $\widehat{\varphi}$ has compact support
\begin{align}\label{eq:analyzing1}
  \supp \widehat{\varphi}
  := \overline{
       \{
         \xi \in \mathbb{R}^d
         :
         \widehat{\varphi} (\xi) \neq 0
       \}
     }
  \subseteq \mathbb{R}^d \setminus \{0\}
\end{align}
and satisfies
\begin{align}\label{eq:analyzing2}
  \sup_{i \in \mathbb{Z}} | \widehat{\varphi}((A^*)^i \xi) | > 0,
  \quad \xi \in \mathbb{R}^d \setminus \{0\}.
\end{align}
In addition to conditions \eqref{eq:analyzing1} and \eqref{eq:analyzing2},
 an $A$-analyzing vector $\varphi$ can be chosen to satisfy
\begin{align}\label{eq:analyzing3}
  \sum_{i \in \mathbb{Z}}
    \widehat{\varphi} ((A^*)^i \xi)
  = 1
  \quad \text{for all} \quad \xi \in \mathbb{R}^d \setminus \{0\},
\end{align}
see, e.g., \cite[Lemma 3.6]{bownik2006atomic} or \cite[Remark 2.3]{cheshmavar2020classification}.
In most situations, we will choose an $A$-analyzing vector that satisfies \eqref{eq:analyzing3}.

\subsection{Triebel-Lizorkin spaces}
\label{sec:TL}
Let $\varphi \in \mathcal{S}(\mathbb{R}^d)$ be a fixed $A$-analyzing vector.
For $i \in \mathbb{Z}$, let $\varphi_i := |\det A|^i \varphi (A^i \cdot)$.
The (homogeneous) \emph{anisotropic Triebel-Lizorkin space} $\TL(A)$,
with $p \in (0, \infty)$, $q \in (0,\infty]$ and $\alpha \in \mathbb{R}$,
is defined as the collection of all $f \in \mathcal{S}' / \mathcal{P}$ satisfying
\begin{align} \label{def:TLspace}
  \| f \|_{\TL(A; \varphi)}
  := \bigg\|
       \bigg(
         \sum_{i \in \mathbb{Z}}
           (|\det A|^{\alpha i} \, |f \ast \varphi_i |)^q
       \bigg)^{1/q}
     \bigg\|_{L^p}
  <  \infty,
\end{align}
with the usual modifications for $q = \infty$.
The space $\TLi(A)$ consists of all $f \in \mathcal{S}' / \mathcal{P}$ such that
\[
  \| f \|_{\TLi(A; \varphi)}
  := \sup_{\ell \in \mathbb{Z}, w \in \mathbb{R}^d}
       \bigg(
         \frac{1}{|\det A|^{\ell}}
         \int_{A^{\ell}\Omega_A + w}
           \sum_{i = -\ell}^{\infty}
             (|\det A|^{\alpha i} \, |(f \ast \varphi_i)(x)|)^q
       \;  d x
       \bigg)^{1/q}
  <  \infty
\]
if $q \in (0,\infty)$, and
\[
  \| f \|_{\TLii(A; \varphi)}
  := \sup_{\ell \in \mathbb{Z}, w \in \mathbb{R}^d}
       \sup_{i \in \mathbb{Z}, i \geq - \ell}
         \frac{1}{|\det A|^{\ell}}
         \int_{A^{\ell}\Omega_A + w}
           |\det A|^{\alpha i} \, |(f \ast \varphi_i)(x)|
       \;  dx
  <  \infty.
\]
In \cite{bownik2006atomic} and the introduction of this paper,
the spaces $\TLi(A)$, $q \in (0, \infty)$,
are alternatively defined using the cube $[0,1]^d$ instead of an expansive ellipsoid $\Omega_A$.
However, it is easily seen that both conditions define the same space, see,
e.g., \cite[Lemma 2.2]{koppensteiner2022anisotropic2}.
See also \Cref{thm:maximal_characterizations} below for the equivalent norm on $\TLii(A)$ used in the introduction.

Each space $\TL(A)$ is continuously embedded into $ \mathcal{S}' / \mathcal{P}$
and is complete with respect to the quasi-norm $\| \cdot \|_{\TL}$.
In addition, $\TL(A)$ is independent of the choice of $A$-analyzing vector $\varphi$,
with equivalent quasi-norms for different choices.
See \cite[Section 3]{bownik2006atomic} and \cite[Section 3.3]{bownik2007anisotropic} for details.
We will often simply write $\| \cdot \|_{\TL(A)}$ for $\| \cdot \|_{\TL(A; \varphi)}$
whenever the precise choice of analyzing vector $\varphi$ does not play a role in our arguments.

For $p, q < \infty$, the space $\Schwartz_0(\R^d)$ is a dense subspace of $\TL(A)$.
This fact follows easily from the various atomic and molecular decompositions of $\TL(A)$,
see, e.g., \cite{bownik2006atomic, bownik2007anisotropic, KvVV2021anisotropic, koppensteiner2022anisotropic2}.

\subsection{Maximal characterizations}
\label{sec:maximal}

For $\varphi \in \mathcal{S} (\mathbb{R}^d)$ and $i \in \mathbb{Z}$ and $\beta > 0$,
the associated Peetre-type maximal function $\varphi^{**}_{i, \beta} f : \mathbb{R}^d \to [0,\infty]$
of $f \in \mathcal{S}' (\mathbb{R}^d)$ is defined by
\[
  \varphi^{**}_{i, \beta} f (x)
  := \sup_{z \in \mathbb{R}^d}
      \frac{|(f \ast \varphi_i) (x+z)|}{(1+\rho_A (A^i z))^{\beta}},
  \quad x \in \mathbb{R}^d.
\]

The following theorem provides characterizations of Triebel-Lizorkin spaces
in terms of Peetre-type maximal functions and will play a key role in \Cref{sec:sufficient}.
See \cite{KvVV2021anisotropic,koppensteiner2022anisotropic2} for proofs.

\begin{theorem}[\cite{KvVV2021anisotropic, koppensteiner2022anisotropic2}] \label{thm:maximal_characterizations}
  Let $A \in \mathrm{GL}(d, \mathbb{R})$ be expansive and $\alpha \in \mathbb{R}$.
  Suppose $\varphi \in \mathcal{S}(\mathbb{R}^d)$ satisfies support conditions \eqref{eq:analyzing1}
  and \eqref{eq:analyzing2}.
  Then the following norm equivalences hold:
  \begin{enumerate}[(i)]
   \item For $p \in (0,\infty)$, $q \in (0, \infty]$ and $\beta > \max \{1/p, 1/q\}$,
         \[
           \| f \|_{\TL(A)}
           \asymp \bigg\|
                    \bigg(
                      \sum_{i \in \mathbb{Z}}
                        (|\det A|^{\alpha i}  \varphi^{**}_{i, \beta} f )^q
                    \bigg)^{1/q}
                  \bigg\|_{L^p},
           \quad f \in \SP,
         \]
    with the usual modification for $q = \infty$.
  \item For $q \in (0,\infty)$ and $\beta > 1/q$,
        \[
          \| f \|_{\TLi(A)}
          \asymp \sup_{\ell \in \mathbb{Z}, w \in \mathbb{R}^d}
                   \bigg(
                     \frac{1}{|\det A|^{\ell}}
                     \int_{A^{\ell} \Omega_A + w}
                       \sum_{i = - \ell}^{\infty}
                         (|\det A|^{\alpha i}  \varphi_{i,\beta}^{**} f (x) )^q
              \;       dx
                   \bigg)^{1/q},
          \quad f \in \SP.
        \]

  \item
        \begin{align}\label{eq:Bii}
           \| f \|_{\TLii(A)}
           \asymp \sup_{i \in \mathbb{Z}}
                    |\det A|^{\alpha i}
                    \| f \ast \varphi_i \|_{L^{\infty}},
           \quad f \in \SP.
        \end{align}
  \end{enumerate}
\end{theorem}

\begin{proof}
  Assertion (i) is part of \cite[Theorem 3.5]{KvVV2021anisotropic}
  and holds for general expansive matrices (cf.\ \cite[Remark 3.6]{KvVV2021anisotropic}).
  Similarly, assertions (ii) and (iii) are part of \cite[Theorem 3.3]{koppensteiner2022anisotropic2}
  and \cite[Theorem 4.1]{koppensteiner2022anisotropic2}, respectively.
\end{proof}

Part (iii) of \Cref{thm:maximal_characterizations} shows that $\TLAii$
coincides with the anisotropic Besov space $\BAii$ considered in \cite{bownik2005atomic}.
In \cite[Definition 3.1]{bownik2005atomic}, the space $\BAii$ is defined via the right-hand side
of the equivalence \eqref{eq:Bii}.

\section{Necessary conditions}
\label{sec:necessary}

This section is devoted to the proof of the following theorem involving necessary conditions
for coincidence of two Triebel-Lizorkin spaces.
This theorem corresponds to \Cref{thm:intro} in the introduction.

\begin{theorem}\label{thm:NecessaryCondition}
  Let $A, B \in \GL(d,\R)$ be expansive matrices, $\alpha, \beta \in \R$ and
  $p_1, p_2, q_1, q_2 \in (0,\infty]$.
  If $\TLone (A) = \TLtwo(B)$, then $(p,q,\alpha) := (p_1,q_1,\alpha) = (p_2,q_2,\beta)$.

  Moreover, at least one of the following two cases holds:
  \begin{enumerate}[(i)]
    \item $A$ and $B$ are equivalent,
          or
    \item $\alpha = 0$, $p \in (1,\infty)$, and $q = 2$.
  \end{enumerate}
\end{theorem}

In the proof of \Cref{thm:NecessaryCondition}, we will often actually
use the norm equivalence
\begin{align} \label{eq:necessary_norm_equiv}
  \| f \|_{\TLone (A)} \asymp \| f \|_{\TLtwo (B)}, \quad \text{for all} \quad f \in \TLone(A) = \TLtwo(B)
\end{align}
rather than the coincidence of the spaces $\TLone (A) = \TLtwo(B)$.
By a standard density argument, the norm equivalence \eqref{eq:necessary_norm_equiv} is equivalent
to the same condition being satisfied for all elements in a dense subspace.
Both facts are contained in the following simple lemma,
which will often be used without further mentioning.

\begin{lemma}\label{lem:CoincidenceImpliesNormEquivalence}
  Let $A, B \in \GL(d,\R)$ be expansive matrices, $\alpha, \beta \in \R$ and
  $p_1, p_2, q_1, q_2 \in (0,\infty]$.

  If $\TLone (A) = \TLtwo (B)$, then there exists a constant
  $C \geq 1$ such that
  \begin{equation}
    \label{eq:CoincidenceImpliesNormEquivalence}
    \frac{1}{C} \| f \|_{\TLone (A)} \leq \| f \|_{\TLtwo(B)} \leq C \| f \|_{\TLone(A)}
  \end{equation}
  for all $f \in \TLone(A) = \TLtwo(B)$.

  On the other hand, if $p_1, p_2, q_1, q_2 < \infty$ and
  \Cref{eq:CoincidenceImpliesNormEquivalence} holds for all
  $f \in \Schwartz_0(\R^d)$, then $\TLone (A) = \TLtwo (B)$.
\end{lemma}

\begin{proof}
  If $\TLone (A) = \TLtwo (B)$, then the identity map
  $\iota: \TLone(A) \to \TLtwo(B), f \mapsto f$ is well-defined.
  Furthermore, since both $\TLone(A)$ and $\TLtwo(B)$ continuously embed into
   $\Schwartz'/\CalP = \Schwartz_0'$ (see \Cref{sec:TL}),
  it is easy to see that $\iota$ has a closed graph.
  The norm estimates \eqref{eq:CoincidenceImpliesNormEquivalence} follow therefore
  by the closed graph theorem, see, e.g., \cite[Theorem~2.15]{RudinFA}.
  More precisely, since, by \cite[Lemma 5.4]{KvVV2021anisotropic} and \cite[Lemma 5.6]{koppensteiner2022anisotropic2}, both $\| \cdot \|_{\TLone(A)}$ and $\| \cdot \|_{\TLtwo(B)}$ are $r$-norms for $r := \min\{p,q,1\}$,
  i.e., both quasi-norms satisfy $\| f_1 + f_2 \|^r \leq \| f_1 \|^r + \| f_2 \|^r$,
  it follows that $\TLone(A)$ and $\TLtwo(B)$ are $F$-spaces in the sense of \cite[Section~1.8]{RudinFA}.
  Therefore,  the closed graph theorem (\mbox{\cite[Theorem~2.15]{RudinFA}}) applies to $\iota$
  and shows that it is bounded, so that
  \[
    \| f \|_{\TLtwo(B)} \lesssim \| f \|_{\TLone(A)}
    \quad \text{for all} \quad
    f \in \TLone(A) = \TLtwo(B).
  \]
  The converse estimate is shown in the same way.

  For the second part of the lemma, recall that $\Schwartz_0(\R^d)$ is norm
  dense in $\TLone(A)$ for $p_1,q_1 < \infty$ (cf.\ \Cref{sec:TL}).
  Hence, for arbitrary $f \in \TLone(A)$, there exists a sequence
  $(f_n)_{n = 1}^\infty$ in $\Schwartz_0(\R^d)$ converging to $f$ in $\TLone(A)$.
  Therefore, if \eqref{eq:CoincidenceImpliesNormEquivalence} holds for all
  $f_n \in \Schwartz_0(\R^d)$, then $(f_n)_{n = 1}^{\infty}$ is a Cauchy sequence
  in $\TLtwo(B)$ converging to some $g \in \TLtwo(B)$.
  Since convergence in $\TLone(A)$, respectively $\TLtwo(B)$, implies weak
  convergence in $\SP$ (cf.\ \Cref{sec:TL}), it follows that $f = g \in \TLtwo(B)$.
  This shows $\TLone(A) \subseteq \TLtwo(B)$.
  The reverse inclusion is shown similarly.
\end{proof}

\subsection{Preparations and notation}%
\label{sub:NecessityProofNotation}

This section sets up some essential objects and notation that will be
used for the proof of \Cref{thm:NecessaryCondition}.
This notation will be kept throughout \Cref{sec:necessary}.

Let $A, B \in \mathrm{GL}(d, \mathbb{R})$ be expansive matrices.
Fix analyzing vectors $\varphi \in \Schwartz(\R^d)$ and
$\psi \in \Schwartz(\R^d)$ satisfying \Cref{eq:analyzing3} for $A$ and $B$, respectively.
Then
\[
  Q := \bigl\{ \xi \in \R^d \colon \widehat{\varphi} (\xi) \neq 0 \bigr\}
  \qquad \text{and} \qquad
  P := \bigl\{ \xi \in \R^d \colon \widehat{\psi}(\xi) \neq 0 \bigr\},
\]
are open, relatively compact sets in $\R^d \setminus \{0\}$.
In the following, we mainly consider the covers $\bigl( (A^\ast)^i Q \bigr)_{i \in \mathbb{Z}}$ and
$\bigl( (B^\ast)^j P \bigr)_{j \in \mathbb{Z}}$ of $\R^d \setminus \{ 0 \}$.
In particular, we will take the sets $I_j$ and $J_i$ defined in \Cref{eq:Ji-Ij} to be defined
with respect to these two coverings; explicitly, this means
\begin{equation}
  I_j := \big\{ k \in \Z \colon (A^\ast)^k Q \cap (B^\ast)^j P \neq \emptyset \big\}
  \quad \text{and} \quad
  J_i := \big\{ k \in \Z \colon (B^\ast)^k P \cap (A^\ast)^i Q \neq \emptyset \big\}
  \label{eq:SpecializedIntersectionSets}
\end{equation}
for $i, j \in \Z$.
Furthermore, for $i \in \mathbb{Z}$, we will use the index sets
\begin{equation*}
  N_i (A^\ast):=  \big\{ j \in \Z \colon (A^\ast)^i Q \cap (A^\ast)^j Q \neq \emptyset \big\}
  \quad \text{and} \quad
  N_i (B^\ast):=  \big\{ j \in \Z \colon (B^\ast)^i P \cap (B^\ast)^j P \neq \emptyset \big\}
  .
\end{equation*}
As shown in \Cref{lem:NeighborSetControl}, there exists $N = N(A,B,Q,P) \in \N$ satisfying
\begin{equation}
  N_i(A^\ast) \cup N_i(B^\ast) \subseteq \{ j \in \Z \colon |j - i| \leq N \}
  \qquad \text{for all} \quad i \in \N .
  \label{eq:NeighborControl}
\end{equation}
Throughout, we fix such an $N$ and define the functions
\[
  \Phi := \sum_{i=-N}^{N} \varphi_i
  \qquad \text{and} \qquad
  \Psi := \sum_{j=-N}^{N} \psi_j
  .
\]
In view of \Cref{eq:analyzing3} and because
$(A^\ast)^i Q \cap Q \neq \emptyset$ can only hold if $|i| \leq N$ by
\Cref{eq:NeighborControl}, it follows that $\widehat{\Phi} \equiv 1$
on $Q$ and $\widehat{\Psi} \equiv 1$ on $P$.
In particular, $\Phi$ and $\Psi$ satisfy the analyzing vector conditions
\eqref{eq:analyzing1} and \eqref{eq:analyzing2} for $A$ and $B$, respectively.

In addition to the above, we fix throughout a non-zero function
$\phi \in \Schwartz(\R^d)$ satisfying $\widehat{\phi} \geq 0$ and
$\supp \widehat{\phi} \subseteq B_1(0)$.
For $\delta > 0$, define
\[
  \phi_\delta (x)
  := \delta^d \, \phi(\delta x).
\]
Then $\widehat{\phi_\delta}(\xi) = \widehat{\phi}(\xi/\delta)$ and thus
$\supp \widehat{\phi_\delta} \subseteq B_\delta(0)$.

\subsection{Norm estimates for auxiliary functions}
\label{sec:norm_estimates}

This subsection consists of two estimates of the Triebel-Lizorkin
norms of functions with specific Fourier support.
These functions play an essential role in our proof of \Cref{thm:NecessaryCondition} and
will be used in the following subsections.

\begin{proposition}
  \label{prop:TL-norm-estimate-simple-Fourier-support}
  Let $A \in \GL(d,\R)$ be expansive, $\alpha \in \R$ and $p, q \in (0, \infty]$.
  If $f \in \Schwartz(\R^d)$ satisfies
  $\supp \widehat{f} \subseteq (A^\ast)^{i_0}Q$ for $i_0 \in \Z$, then
  \begin{equation}\label{eq:TL-norm-estimate-simple-Fourier-support}
    \|f\|_{\TL(A)} \asymp |\det A|^{\alpha i_0} \|f\|_{L^p},
  \end{equation}
  with an implicit constant independent of $i_0$ and $f$.
\end{proposition}

\begin{proof}
  With notation as in \Cref{sub:NecessityProofNotation}, we start by collecting some
  basic facts about the convolutions $f \ast \varphi_i$ and
  $f \ast \Phi_{i_0}$ for $f$ as in the statement of the proposition.
  First, note that since $\widehat{\varphi_i} \equiv 0$ outside of
  $(A^\ast)^{i}Q$, it follows that $f \ast \varphi_i = 0$ whenever
  $(A^\ast)^{i_0}Q \cap (A^\ast)^{i}Q = \emptyset$,
  which holds whenever $|i-i_0| > N$, by \Cref{eq:NeighborControl}.
  Therefore,
  \begin{equation}
    \label{eq:lem1-f-conv-varphi}
    f \ast \varphi_i \equiv 0 \qquad \text{for }|i-i_0|> N.
  \end{equation}
  For the convolution  $f \ast \Phi_{i_0}$ observe that $\widehat{\Phi_{i_0}} \equiv 1$ on
  $(A^\ast)^{i_0}Q$ by construction, and therefore
  \begin{equation}
    \label{eq:lem1-f-conv-Phi}
    f \ast \Phi_{i_0} = \Fourier^{-1}(\widehat{f} \cdot
  \widehat{\Phi_{i_0}}) = f.
  \end{equation}

  In the remainder of this proof, we deal with the cases $p < \infty$,
  $p = \infty$ and $q < \infty$, and $p = q = \infty$ separately.
  \\~\\
  \textbf{Case 1:} $p \in (0, \infty)$.
  For the upper bound in \Cref{eq:TL-norm-estimate-simple-Fourier-support}, we use
  \eqref{eq:lem1-f-conv-varphi} to obtain
  \begin{equation*}
    \| f \|_{\TL(A;\varphi)}
     = \bigg\|
          \Big\|
            \Big( |\det A|^{\alpha i}
              |f \ast \varphi_i|
            \Big)_{i \in \Z}
          \Big\|_{\ell^q}
        \bigg\|_{L^p}
     \lesssim_{p,q,N} \sum_{i = i_0-N}^{i_0 + N} |\det A|^{\alpha i} \| f \ast \varphi_i \|_{L^p}.
   \end{equation*}
   If $p \in [1,\infty)$, then Young's inequality shows
   \begin{equation*}
     \| f \ast \varphi_i \|_{L^p}
     \leq \| f \|_{L^p}  \| \varphi_i \|_{L^1}
     \lesssim_{\varphi} \| f \|_{L^p}.
   \end{equation*}
   If $p \in (0,1)$, then, since
   \(
     \supp \widehat{f}, \supp \widehat{\varphi_i}
     \subseteq \bigcup_{\ell = -N}^N (A^\ast)^{i_0 + \ell} \overline{Q}
   \)
   for $|i - i_0| \leq N$, an application of \Cref{cor:ConvenientConvolutionRelation} yields
   \begin{align*}
    \| f \ast \varphi_i \|_{L^p}
    &\lesssim_{A,Q,N,p} |\det A|^{i_0  (\frac{1}{p} - 1)}
                        \| f \|_{L^p}
                        \| \varphi_i \|_{L^p} \\
    &= |\det A|^{(i_0 - i) (\frac{1}{p} - 1)} \| f \|_{L^p} \| \varphi \|_{L^p} \\
    &\lesssim_{A,N,\varphi,p} \| f \|_{L^p}.
  \end{align*}
  Consequently, for arbitrary $p \in (0,\infty]$,
  \begin{equation*}
    \| f \|_{\TL(A;\varphi)}
    \lesssim_{p,q,N} \sum_{i = i_0-N}^{i_0 + N} |\det A|^{\alpha i} \| f \ast \varphi_i \|_{L^p}
    \lesssim_{A,Q,N, \alpha, p,\varphi} |\det A|^{\alpha i_0} \| f \|_{L^p},
  \end{equation*}
  which proves the desired upper bound.

  For the lower bound, using \Cref{eq:lem1-f-conv-Phi} and the equivalence
  $\| \cdot \|_{\TL(A; \varphi)} \asymp \| \cdot \|_{\TL(A; \Phi)}$ (see \Cref{sec:TL}) gives
  \begin{align*}
    \| f \|_{\TL (A;\varphi)}
    & \asymp_{\varphi,N,p,q,A,\alpha} \| f \|_{\TL (A;\Phi)}
    = \bigg\|
        \Big\|
        \Big(
        |\det A|^{\alpha i}
            |f \ast \Phi_i|
          \Big)_{i \in \Z}
        \Big\|_{\ell^q}
      \bigg\|_{L^p} \\
    & \geq |\det A|^{\alpha i_0} \| f \ast \Phi_{i_0} \|_{L^p}
    = |\det A|^{\alpha i_0} \| f \|_{L^p},
  \end{align*}
  as required.
  \\~\\
  \textbf{Case 2:} $p = \infty$, $q \in (0, \infty)$.
  As in the previous case, we use \Cref{eq:lem1-f-conv-varphi} for the upper estimate.
  This yields
  \begin{align*}
    \|f\|_{\TLi(A;\varphi)}
    &= \sup_{\ell \in \Z, w \in \R^d}
      \bigg( \frac{1}{|\det A|^{\ell}} \int_{A^{\ell} \Omega_A + w}
      \sum_{i = - \ell}^{\infty} (|\det A|^{\alpha i}
      |(f \ast \varphi_i)(x)|)^q \; dx \bigg)^{1/q} \\
    &\leq \sup_{\ell \in \Z, w \in \R^d}
      \bigg( \frac{1}{|\det A|^{\ell}} \int_{A^{\ell} \Omega_A + w}
      \sum_{i = i_0 -N}^{i_0+N} (|\det A|^{\alpha i}
      |(f \ast \varphi_i)(x)|)^q \; dx \bigg)^{1/q} \\
    &\lesssim_{A, N, q, \alpha} |\det A|^{\alpha i_0}  \sum_{i = i_0 -N}^{i_0+N}
      \| f \ast \varphi_i\|_{L^\infty} \\
    &\leq |\det A|^{\alpha i_0}  \sum_{i = i_0 -N}^{i_0+N}
      \| f \|_{L^\infty}  \| \varphi_i\|_{L^1} \\
    & \lesssim_{N, \varphi}  |\det A|^{\alpha i_0} \| f \|_{L^\infty}.
  \end{align*}
  For the lower bound, we use the continuous embedding
  $\TLi(A) \hookrightarrow \TLii(A)$ (cf.\ \cite[Theorem 4.1]{koppensteiner2022anisotropic2}),
  the norm equivalence $\| \cdot \|_{\TLii(A; \varphi)} \asymp \| \cdot \|_{\TLii(A; \Phi)}$,
  and \Cref{eq:Bii} to obtain
  \begin{align*}
    \|f\|_{\TLi(A;\varphi)}
    &\gtrsim_{\varphi,q,A,\alpha} \|f\|_{\TLii(A;\varphi)}
    \asymp_{\varphi, N, A,\alpha} \|f\|_{\TLii(A;\Phi)}
    \asymp_{\varphi,N,A,\alpha} \sup_{i \in \Z} |\det A|^{\alpha i} \| f \ast \Phi_i \|_{L^\infty} \\
    &\geq |\det A|^{\alpha i_0} \| f \ast \Phi_{i_0} \|_{L^\infty}
    = |\det A|^{\alpha i_0} \| f \|_{L^\infty},
  \end{align*}
  where the final step follows from \Cref{eq:lem1-f-conv-Phi}.
  \\~\\
  \textbf{Case 3:} $p =q = \infty$.
  The lower bound $ \|f\|_{\TLii(A;\varphi)} \gtrsim |\det A|^{\alpha i_0} \| f
  \|_{L^\infty}$ has been shown in the previous case already.
  For the reverse, observe that \eqref{eq:Bii} and \eqref{eq:lem1-f-conv-varphi} yield
  \begin{align*}
    \|f\|_{\TLii(A;\varphi)}
    &\asymp \sup_{i \in \Z} |\det A|^{\alpha i} \| f \ast \varphi_i \|_{L^\infty}
    = \sup_{|i - i_0| \leq N} |\det A|^{\alpha i} \| f \ast \varphi_i \|_{L^\infty} \\
    &\leq \sup_{|i - i_0| \leq N} |\det A|^{\alpha i} \| f\|_{L^\infty} \dot \|\varphi_i \|_{L^1}
    \lesssim_{N, A, \alpha, \varphi} |\det A|^{\alpha i_0} \| f\|_{L^\infty},
  \end{align*}
  which completes the proof.
\end{proof}

The following simple consequence is what actually will be used in obtaining necessary conditions
for the coincidence of Triebel-Lizorkin spaces.

\begin{corollary}
  \label{cor:detA-equiv-detB}
  Let $A, B \in \GL(d,\R)$ be expansive, $\alpha, \beta \in \R$ and
  $p_1, p_2, q_1, q_2 \in (0,\infty]$.

  Suppose that $\TLone(A) = \TLtwo(B)$.
  If $(A^\ast)^i Q \cap (B^\ast)^j P \neq \emptyset$ for some $i, j \in \mathbb{Z}$,
  then there exists $\delta_0 = \delta_0(i,j) > 0$ such that for all $0 < \delta \leq \delta_0$,
  it holds that
  \[
    |\det A|^{\alpha i}  \delta^{d  (1-1/p_1)}
    \asymp |\det B|^{\beta j}  \delta^{d  (1-1/p_2)},
  \]
  where the implicit constants are independent of $i, j, \delta, \delta_0$.
\end{corollary}

\begin{proof} Since $(A^\ast)^i Q \cap (B^\ast)^j P \neq \emptyset$ is
  open, there exists $\eta \in \R^d$ and  $\delta_0 > 0$ such that
  $B_{\delta_0}(\eta) \subseteq (A^\ast)^i Q \cap (B^\ast)^j P$.
  For a fixed $0 < \delta \leq \delta_0$, define
  $f_\delta := M_\eta \phi_\delta$.
  Then
  \[
    \supp \widehat{f_\delta}
    = \supp T_\eta \widehat{\phi_\delta}
    \subseteq B_{\delta}(\eta).
  \]
  Using the estimates of
  \Cref{prop:TL-norm-estimate-simple-Fourier-support} for
  $\|f_\delta\|_{\TLone(A)}$ and $\|f_\delta\|_{\TLtwo(B)}$ yields
  \begin{align*}
    |\det A|^{\alpha i}  \delta^{d  (1-1/p_1)}
    = |\det A|^{\alpha i} \, \|f_\delta\|_{L^{p_1}}
      \asymp \|f_\delta\|_{\TLone(A)}
     \asymp \|f_\delta\|_{\TLtwo(B)}
      \asymp |\det B|^{\beta j}  \delta^{d  (1-1/p_2)},
  \end{align*}
  with implicit constants independent of $i,j, \delta, \delta_0$.
\end{proof}

The following proposition provides a more technical version
of Proposition~\ref{prop:TL-norm-estimate-simple-Fourier-support}
and involves a linear combination of functions with Fourier supports in $(A^*)^{i_k} Q$
for suitable points $i_k \in \mathbb{Z}$.
The proof strategy resembles the one of \Cref{prop:TL-norm-estimate-simple-Fourier-support},
but requires various technical modifications.

\begin{proposition}\label{prop:TL-norm-estimate-sequence}
  Let $A \in \GL(d,\R)$ be expansive, $\alpha \in \R$ and $p, q \in (0, \infty]$.
  For $K \in \N$, let $i_1, \dots, i_K \in \Z$ be
  increasing with $|i_k - i_{k'}| > 2N$ if $k \neq k'$, where
  $N \in \N$ is as in \Cref{eq:NeighborControl}.

  Suppose there exists $\delta_0 > 0$ and points
  $\eta_1, \dots, \eta_K \in \R^d$ such that:
  \begin{enumerate}
   \item[(a)] $B_{\delta_0}(\eta_k) \subseteq (A^\ast)^{i_k}Q$ for all $k=1, \dots, K$,
   \item[(b)] $|\phi(x)| \geq \frac{1}{2}|\phi(0)|$ for all $x \in \delta_0 A^{-i_1}\Omega_A$.
  \end{enumerate}
  Then, for all $0 < \delta \leq \delta_0$ and $c \in \CC^K$, the function
  $f = \sum_{k=1}^K c_k M_{\eta_k} \phi_\delta$ satisfies
   \begin{equation} \label{eq:TL-norm-estimate-sequence}
    \| f \|_{\TL(A)} \asymp \delta^{d (1 -1/p)}
    \Big\|\Big(|\det A|^{\alpha i_k} \, |c_k|\Big)_{k =1}^K \Big\|_{\ell^q},
  \end{equation}
  where the implicit constant is independent of
  $K, c, \delta, \delta_0, \eta_1, \dots, \eta_K, i_1, \dots, i_K$.
\end{proposition}

\begin{remark}
  Assumption (b) in \Cref{prop:TL-norm-estimate-sequence} is only needed for the case
  $p = \infty$, $q \in (0, \infty)$.
\end{remark}

\begin{proof}
  Using the notation from \Cref{sub:NecessityProofNotation}, we first state some basic
  observations for $f \ast \varphi_i$ and $f \ast \Phi_{i_k}$ with $f$ as in the statement.
  First, note that by assumption (a) it follows that
  $\supp T_{\eta_k}\widehat{\phi_\delta} \subseteq (A^\ast)^{i_k}Q$ for all $k=1, \dots, K$.
  Since $\widehat{\varphi_i} \equiv 0$ outside of $(A^\ast)^{i}Q$, this implies
  \begin{equation*}
    M_{\eta_k}[\phi_\delta] \ast \varphi_i
    = \Fourier^{-1}(T_{\eta_k} [\widehat{\phi_\delta}] \cdot \widehat{\varphi_i})
    = 0 \qquad \text{for } |i - i_k| > N,
  \end{equation*}
  as $(A^\ast)^{i_k}Q \cap (A^\ast)^{i}Q = \emptyset$ for
  $|i - i_k| > N$ by \Cref{eq:NeighborControl}.
  Furthermore, note that for fixed $i \in \Z$, there can be at most one point
  $i_k$ such that $|i - i_k| \leq N $ due to the pairwise minimal distance
  between the chosen points $i_1, ..., i_K$.
  This implies that
  \begin{equation}\label{eq:lem2-f-conv-varphi}
    f \ast \varphi_i
    = \sum_{k=1}^K c_k \cdot (M_{\eta_k} [\phi_\delta] \ast \varphi_i)
    = \begin{cases}
         c_k \cdot (M_{\eta_k} [\phi_\delta] \ast \varphi_i), & \text{if } |i - i_k| \leq N,\\
         0, & \text{otherwise.}
      \end{cases}
  \end{equation}

  Second, for $f \ast \Phi_{i_k}$, observe that $\widehat{\Phi_{i_k}} \equiv 0$ outside of
  $\bigcup_{i= i_k -N}^{i_k+N}(A^\ast)^{i}Q$ for all $k = 1, \dots, K$ by construction of $\Phi$.
  Since $|i_k - i_{k'}| > 2N$ for $k \neq k'$, it follows by \Cref{eq:NeighborControl} that
  \[
    (A^\ast)^{i_{k'}}Q \cap \bigcup_{i= i_k -N}^{i_k+N}(A^\ast)^{i}Q
    = \emptyset,
    \quad \text{for} \quad k \neq k'.
  \]
  This implies
  \(
    M_{\eta_{k'}} [\phi_\delta] \ast \Phi_{i_k}
    = \Fourier^{-1}(T_{\eta_{k'}} [\widehat{\phi_\delta}] \cdot \widehat{\Phi_{i_k}})
    = 0
  \)
  for $k \neq k'$.
  Since also $\widehat{\Phi_{i_k}} \equiv 1$ on
  $(A^\ast)^{i_k}Q \supseteq \supp T_{\eta_k}\widehat{\phi_\delta}$, necessarily
  \begin{equation}
    \label{eq:lem2-f-conv-Phi}
    f \ast \Phi_{i_k}
    = \sum_{k'=1}^K c_{k'} \cdot (M_{\eta_{k'}} [\phi_\delta] \ast \Phi_{i_k})
    = c_kM_{\eta_k} \phi_\delta \qquad \text{for } k= 1, \dots K.
  \end{equation}

  The remainder of the proof is divided into three cases and deals
  with $p < \infty$, $p = \infty$ and $q < \infty$, and
  $p = q = \infty$ separately.
  \\~\\
  \textbf{Case 1:} $p \in (0, \infty)$.
  For the upper bound in \Cref{eq:TL-norm-estimate-sequence}, set $M = \frac{d}{p} + 1$.
  Then, in view of \Cref{eq:lem2-f-conv-varphi}, an application of
  \Cref{lem:ImprovedConvolutionBound} with $\ell = i_k$ shows that
  \begin{align*}
    |f \ast \varphi_i (x)|
    & = |c_k| \cdot |(M_{\eta_k} [\phi_\delta] \ast \varphi_i)(x)|
      \leq |c_k| \cdot (|\phi_\delta| \ast |\varphi_i|) (x) \\
    & \lesssim_{N,A,d,p,Q,\phi,\varphi} |c_k|  \delta^d  (1 + |\delta x|)^{-M}
  \end{align*}
  whenever $| i - i_k| \leq N$.
  On the other hand, $f \ast \varphi_i = 0$  if $|i - i_k| > N$ for all $k = 1,\dots,K$.
  Therefore, for all $x \in \R^d$,
  \begin{align*}
    \Big\|
      \Big(
        |\det A|^{\alpha i}
        |f \ast \varphi_i(x)|
      \Big)_{i \in \Z}
    \Big\|_{\ell^q}
    & \leq \bigg(
             \sum_{k = 1}^{K}
               \sum_{i= i_k -N}^{i_k +N}
                 (|\det A|^{\alpha i} \, |f \ast \varphi_i(x)|)^q
           \bigg)^{1/q} \\
    & \lesssim_{N,A,Q,d,p,q,\phi,\varphi,\alpha}
           \bigg(
             \sum_{k = 1}^{K}
             \big(
               |\det A|^{\alpha i_k} \,
               |c_k| \,
               \delta^d \,
               (1 + |\delta x|)^{-M}
             \big)^q
           \bigg)^{1/q} \\
    & = \delta^d \, (1 + |\delta x|)^{-M} \,
      \Big\|\Big(|\det A|^{\alpha i_k}  |c_k|\Big)_{k =1}^K \Big\|_{\ell^q},
  \end{align*}
  with the usual modification of the argument for $q = \infty$.
  Consequently, this yields
  \begin{align*}
    \| f \|_{\TL(A;\varphi)}
    & = \bigg\|
          \Big\|
            \Big(
              |\det A|^{\alpha i}
              |f \ast \varphi_i|
            \Big)_{i \in \Z}
          \Big\|_{\ell^q}
        \bigg\|_{L^p} \\
    & \lesssim_{N,A,Q,d,p,q,\phi,\varphi,\alpha}
        \bigg(
          \int_{\R^d} (\delta^d \, (1 + |\delta x|)^{-M})^p \, dx
        \bigg)^{1/p}
        \Big\|
          \Big(
            |\det A|^{\alpha i_k}  |c_k|
          \Big)_{k =1}^K
        \Big\|_{\ell^q} \\
    & \lesssim_{d,p} \delta^{d  (1- 1/p)} \,
      \Big\|\Big(|\det A|^{\alpha i_k}  |c_k|\Big)_{k =1}^K \Big\|_{\ell^q},
  \end{align*}
  where the last step used that $M > \frac{d}{p}$,
  so that $\int_{\R^d} (1 + |x|)^{-Mp} \, dx < \infty$.

  For the lower bound, we use the equivalence
  $\| \cdot \|_{\TL(A; \varphi)} \asymp \| \cdot \|_{\TL(A; \Phi)}$ and \Cref{eq:lem2-f-conv-Phi}
  to obtain
  \begin{align*}
    \| f \|_{\TL(A;\varphi)}
    &\asymp_{A,p,q,\alpha,\varphi,N}  \bigg\|
          \Big\|
            \Big(
              |\det A|^{\alpha i} \,
              |f \ast \Phi_i|
            \Big)_{i \in \Z}
          \Big\|_{\ell^q}
        \bigg\|_{L^p} \\
    & \geq \bigg\|
             \Big\|
               \Big(
                 |\det A|^{\alpha i_k} \,
                 |f \ast \Phi_{i_k}|
               \Big)_{k =1}^K
             \Big\|_{\ell^q}
           \bigg\|_{L^p} \\
    & = \| \phi_\delta \|_{L^p}
       \Big\|
         \Big(
           |\det A|^{\alpha i_k} \, |c_k|
         \Big)_{k =1}^K
       \Big\|_{\ell^q} \\
    & \gtrsim_{\phi,p}
        \delta^{d  (1- 1/p)}
        \Big\|
          \Big(
            |\det A|^{\alpha i_k} \, |c_k|
          \Big)_{k =1}^K
        \Big\|_{\ell^q},
  \end{align*}
  as required.
  \\~\\
  \textbf{Case 2:} $p = \infty$, $q \in (0, \infty)$.
  The upper estimate in \Cref{eq:TL-norm-estimate-sequence}
  follows by an application of \Cref{eq:lem2-f-conv-varphi}:
  \begin{align*}
    \|f\|_{\TLi(A;\varphi)}
    &\leq \sup_{\ell \in \Z, w \in \R^d}
            \bigg(
              \frac{1}{|\det A|^{\ell}}
              \int_{A^{\ell} \Omega_A +w}
                \sum_{k= 1}^{K}
                  \sum_{i=i_k -N}^{i_k +N}
                  \big(
                    |\det A|^{\alpha i} \,
                    |c_k| \,
                    (|\phi_\delta| \ast |\varphi_{i}|) (x)
                  \big)^q
             \; dx
            \bigg)^{1/q} \\
    & \leq \bigg(
             \sum_{k= 1}^{K}
               \sum_{i=i_k -N}^{i_k +N}
               \big(
                 |\det A|^{\alpha i} \,
                 |c_k| \,
                 \| |\phi_\delta| \ast |\varphi_{i}| \|_{L^\infty}
               \big)^q
           \bigg)^{1/q}  \\
    & \leq \bigg(
             \sum_{k= 1}^{K}
               \sum_{i=i_k -N}^{i_k +N}
               \big(
                 |\det A|^{\alpha i} \,
                 |c_k| \,
                 \| \phi_\delta\|_{L^\infty} \,
                 \| \varphi_{i} \|_{L^1}
               \big)^q
           \bigg)^{1/q}  \\
    & \lesssim_{A, N, q, \alpha} \delta^d \, \|\phi\|_{L^\infty}  \|\varphi\|_{L^1}
      \bigg(\sum_{k= 1}^{K}(|\det A|^{\alpha i_k}  |c_k|  )^q \bigg)^{1/q}.
  \end{align*}
  For the reverse inequality, we again use the $A$-analyzing vector $\Phi$.
  We start by taking $w = 0$ and $\ell = - i_1$ in the supremum below.
  Note that this choice ensures that the sum over
  $i \geq - \ell = i_1$ includes all $i_k$ for $k=1, \dots, K$ as they are increasing.
  By \Cref{eq:lem2-f-conv-Phi}, it follows that
  \begin{align*}
    \|f\|_{\TLi(A;\varphi)}
    & \asymp_{A,q,\alpha,\varphi,N}
      \sup_{\ell \in \Z, w \in \R^d}
        \bigg(
          \frac{1}{|\det A|^{\ell}}
          \int_{A^{\ell} \Omega_A + w}
            \sum_{i = - \ell}^{\infty}
              (|\det A|^{\alpha i} \, |(f \ast \Phi_i) (x)|)^q
          \, dx
        \bigg)^{1/q} \\
    &\geq \bigg(
            \frac{1}{|\det A|^{-i_1}}
            \int_{A^{-i_1} \Omega_A}
              \sum_{k= 1}^{K}
              (|\det A|^{\alpha i_k} \, |(f \ast \Phi_{i_k}) (x)|)^q
            \, dx
          \bigg)^{1/q} \\
    &= \bigg(
         \frac{1}{|\det A|^{-i_1}}
         \int_{A^{-i_1} \Omega_A}
           \sum_{k= 1}^{K}
           (|\det A|^{\alpha i_k} \, |c_k| \, |\phi_\delta(x)|)^q
         \, dx
       \bigg)^{1/q}\\
    &\geq  \min_{x \in \delta A^{-i_1} \Omega_A}
             \delta^d \, |\phi(x)|  \,
             \Big\|
               \Big(
                 |\det A|^{\alpha i_k} \, |c_k|
               \Big)_{k =1}^K
             \Big\|_{\ell^q} \\
    &\geq \delta^d \, \frac{1}{2} \, |\phi(0)|
          \Big\|
            \Big(
              |\det A|^{\alpha i_k} \, |c_k|
            \Big)_{k =1}^K
          \Big\|_{\ell^q},
  \end{align*}
  where we used the assumption $|\phi(x)| \geq \frac{1}{2} |\phi(0)|$
  for $x \in \delta_0 A^{-i_1} \Omega_A$ in the last step.
  Furthermore, note that $\phi(0) > 0$ since $\widehat{\phi} \geq 0$
  and $\widehat{\phi} \not\equiv 0$, so that
  $\phi(0) = \int_{\R^d} \widehat{\phi}(\xi) \, d \xi > 0$.
  \\~\\
  \textbf{Case 3:} $p = q = \infty$.
  \Cref{eq:Bii,eq:lem2-f-conv-varphi} allow to obtain the upper bound:
  \begin{align*}
    \|f\|_{\TLii(A;\varphi)}
    & \asymp \sup_{i \in \Z}
               |\det A|^{\alpha i}
               \| f \ast \varphi_i \|_{L^\infty} \\
      &= \sup_{k=1, \dots, K} \,\,
           \sup_{|i - i_k| \leq N}
             |\det A|^{\alpha i}
             \| f \ast \varphi_i \|_{L^\infty}\\
    &\leq \sup_{k=1, \dots, K} \,\,
            \sup_{|i - i_k| \leq N}
              |\det A|^{\alpha i} \,
              |c_k| \,
              \|\phi_\delta\|_{L^\infty} \,
              \|\varphi_{i}\|_{L^1} \\
    & \lesssim_{N, A, \alpha, \phi, \varphi}
      \delta^d
      \Big\|
        \Big(
          |\det A|^{\alpha i_k} \, |c_k|
        \Big)_{k = 1}^K
      \Big\|_{\ell^{\infty}}.
  \end{align*}
  For the lower bound, combining the norm equivalence
  $\| \cdot \|_{\TLii(A; \varphi)} \asymp \| \cdot \|_{\TLii(A; \Phi)}$
  and \Cref{eq:Bii,eq:lem2-f-conv-Phi} yields
  \begin{align*}
    \|f\|_{\TLii(A;\varphi)}
    & \asymp_{A,\varphi,N,\alpha} \|f\|_{\TLii(A;\Phi)}
      \asymp \sup_{i \in \Z}
               |\det A|^{\alpha i} \,
               \| f \ast \Phi_i \|_{L^\infty} \\
    & \geq \sup_{k=1, \dots, K}
             |\det A|^{\alpha i_k} \,
             \| f \ast \Phi_{i_k} \|_{L^\infty}
    = \sup_{k=1, \dots, K}
       |\det A|^{\alpha i_k} \,
       |c_k| \,
       \| \phi_\delta \|_{L^\infty} \\
    & \gtrsim_{\phi} \delta^d
      \Big\|\Big(|\det A|^{\alpha i_k}  |c_k|\Big)_{k =1}^L \Big\|_{\ell^{\infty}},
  \end{align*}
  which completes the proof.
\end{proof}

\subsection{The case \texorpdfstring{$\alpha \neq 0$}{α ≠ 0}}

This subsection is devoted to the proof of the following theorem.
In particular, it shows that two expansive matrices $A, B \in \mathrm{GL}(d, \mathbb{R})$ are equivalent
whenever $\dot{\mathbf{F}}_{p,q}^{\alpha} (A) = \dot{\mathbf{F}}_{p,q}^{\alpha} (B)$ and $\alpha \neq 0$.
This proves \Cref{thm:NecessaryCondition} for the case $\alpha \neq 0$.

\begin{theorem}
  Let $A, B \in \GL(d, \R)$ be expansive, $\alpha, \beta \in \R$ and
  $p_1,p_2, q_1, q_2 \in (0,\infty]$.
  If $\TLone (A) = \TLtwo(B)$, then the following hold:
  \begin{enumerate}[(i)]
    \item $p_1 = p_2$,

    \item $q_1 = q_2$,

    \item $\alpha = \beta$.
          Furthermore, if $\alpha = \beta \neq 0$, then $A$ and $B$ are equivalent.
  \end{enumerate}
\end{theorem}

\begin{proof} We prove the three assertions separately.
\\~\\
 (i) Since $\varphi, \psi \in \Schwartz(\R^d)$
  are analyzing vectors for $A$ resp.\ $B$, it follows that
  \begin{align} \label{eq:covers_analyzing}
   \bigcup_{i \in \Z} (A^\ast)^{i} Q = \bigcup_{j \in \Z}
  (B^\ast)^{j} P = \mathbb{R}^d \setminus \{0\} .
  \end{align}
  Hence, there exist $i_0, j_0 \in \Z$ such that
  $(A^\ast)^{i_0} Q \cap (B^\ast)^{j_0} P \neq \emptyset$.
  By \Cref{cor:detA-equiv-detB}, this implies the existence of some $\delta_0 > 0$ such that
  for, all $0 < \delta \leq \delta_0$,
    \begin{equation*}
     |\det A|^{\alpha i_0}  \delta^{d  (1-1/p_1)}
    \asymp |\det B|^{\beta j_0}  \delta^{d  (1-1/p_2)},
  \end{equation*}
  with implicit constant independent of $\delta$.
  In turn, this implies
  \begin{equation*}
    \delta^{1/p_1 - 1/p_2} \asymp 1, \qquad \text{for all} \quad 0 < \delta \leq \delta_0,
  \end{equation*}
  which is only possible for $p_1 = p_2$.
  \\~\\
  (ii) Under the assumption $p_1 = p_2 = p$, we show that
  \begin{equation}
    \| c \|_{\ell^{q_1}} \asymp \| c \|_{\ell^{q_2}},
    \qquad \text{for all} \quad \, K \in \N, \; c \in \CC^K
    ,
    \label{eq:QIdentificationSequenceNorm}
  \end{equation}
  where the implied constant is independent of $K$ and $c$.
  This easily implies $q_1 = q_2$.

  Let $K \in \N$ be arbitrary and let $N \in \N$ be as chosen in \Cref{eq:NeighborControl}.
  Recall the identity \eqref{eq:covers_analyzing}
  and note that each image
  set $(A^\ast)^{i} Q$, $(B^\ast)^{j} P$ for $i,j \in \Z$ is
  relatively compact and hence bounded.
  Therefore, it is not hard to see that there exist points $\eta_1,\dots,\eta_K \in \R^d$
  and increasing sequences $(i_k)_{k = 1}^K$ and $(j_k)_{k = 1}^K$ in $\Z$
  satisfying
  \begin{equation*}
    |i_k - i_{k'}| > 2N \quad \text{and} \quad |j_k - j_{k'}| > 2N
    \qquad \text{for } k \neq k',
  \end{equation*}
   with
  \begin{equation}
    \label{eq:QIdentificationXiInIntersection}
    \eta_k \in (A^\ast)^{i_k} Q \cap (B^\ast)^{j_k} P
    \quad \text{for all} \quad
    k = 1,\dots,K.
  \end{equation}
 Since the sets $Q,P$ are open, there exists $\delta_1 > 0$ such that
  \begin{equation*}
    B_{\delta_1}(\eta_k)
    \subseteq (A^\ast)^{i_k} Q \cap (B^\ast)^{j_k} P
    \quad \text{for all} \quad
    k = 1, \dots, K.
  \end{equation*}
  Additionally, by continuity of $\phi \in \Schwartz(\R^d)$,
  there exists $\delta_2 > 0$ such that
  \begin{equation*}
    |\phi(x)| \geq \frac{1}{2} |\phi(0)|,
    \qquad  x \in \delta_2 ( A^{-i_1}\Omega_A \cup B^{-j_1}\Omega_B).
  \end{equation*}
  In combination, this shows that the assumptions of \Cref{prop:TL-norm-estimate-sequence} are
  met for $\TLgeneral{p}{q_1}{\alpha}(A;\varphi)$ with
  $\delta_0 := \min \{\delta_1, \delta_2\}$, $\eta_1,\dots,\eta_K$
  and $(i_k)_{k = 1}^K$, as well as for
  $\TLgeneral{p}{q_2}{\beta}(B;\psi)$ with $(j_k)_{k = 1}^K$ replacing
  the sequence $(i_k)_{k = 1}^K$.

  For showing the claim \eqref{eq:QIdentificationSequenceNorm},
  let $c \in \CC^K$ and $0 < \delta \leq \delta_0$ be fixed.
  Then defining
  $f_\delta := \sum_{k=1}^K |\det A|^{- \alpha i_k} c_k M_{\eta_k}
  \phi_\delta $ gives
  \begin{align*}
    \delta^{d  (1 -1/p)}
    \| c \|_{\ell^{q_1}}
    &\asymp \|f_\delta\|_{\TLgeneral{p}{q_1}{\alpha}(A)} \\
     & \asymp \|f_\delta\|_{\TLgeneral{p}{q_2}{\beta}(B)} \\
    &\asymp \delta^{d  (1 -1/p)}
      \Big\|\Big(|\det B|^{\beta j_k}\, |\det A|^{- \alpha i_k}\, |c_k|
      \Big)_{k =1}^K\Big\|_{\ell^{q_2}}.
      \numberthis   \label{eq:q1=q2}
  \end{align*}
  Since $(A^\ast)^{i_k} Q \cap (B^\ast)^{j_k} P \neq \emptyset$ by
  \Cref{eq:QIdentificationXiInIntersection}, it follows by
  \Cref{cor:detA-equiv-detB} for $p_1 = p_2 = p$ that
  $|\det A|^{\alpha i_k}  \asymp |\det B|^{\beta j_k}$ for $ k=1, \dots, K$,
  with implicit constant independent of $i_k, j_k$.
  This, together
  with \Cref{eq:q1=q2}, easily shows the claim \eqref{eq:QIdentificationSequenceNorm}.
  \\~\\
  (iii) Assuming $p_1 = p_2 = p$, it follows by
  \Cref{cor:detA-equiv-detB} that there exists $C \geq 1$
  such that
  \begin{equation}
    \label{eq:WeightEquivalenceExplicit}
    \frac{1}{C}  |\det B|^{\beta j}
    \leq |\det A|^{\alpha i}
    \leq C  |\det B|^{\beta j}
    \quad \text{whenever} \quad
    (A^\ast)^i Q \cap (B^\ast)^j P \neq \emptyset.
  \end{equation}
  We consider the cases $\alpha =0$ or $\beta = 0$, and $\alpha \neq 0 \neq \beta$.

  \textit{Case~1: $\alpha = 0$ or $\beta = 0$.}
  Suppose first that $\alpha = 0$.
  As a consequence of \Cref{eq:covers_analyzing},
  for all $j \in \Z$, there needs to exist $i \in \Z$ such that
  $(A^\ast)^i Q \cap (B^\ast)^j P \neq \emptyset$.
  \Cref{eq:WeightEquivalenceExplicit} implies therefore that
  $|\det B|^{\beta j} \leq C$ as $\alpha = 0$.
  Since this holds for all $j \in \Z$, and $|\det B| \neq 0$,
  it follows that necessarily also $\beta = 0$.
  If $\beta = 0$, then also $\alpha = 0$ by symmetry.

  \emph{Case~2: $\alpha \neq 0 \neq \beta$.}
  Suppose that $\alpha \neq 0 \neq \beta$.
  Then, by \Cref{eq:WeightEquivalenceExplicit},
  the assumptions of \Cref{lem:IjJiInclusions-with-alpha-beta}
  are satisfied for $\big( (A^\ast)^i Q \big)_{i \in \Z}$ and
  $\big( (B^\ast)^j P \big)_{j \in \Z}$.
  Hence, there exists $M \in \N$ such that with $J_i,I_j$
  as defined in \Cref{eq:SpecializedIntersectionSets}, we have
  \[
    J_i \subseteq \Big\{ j \in \Z \colon \Big|j - \Big\lfloor
    \frac{\alpha}{\beta} \, c \, i \Big\rfloor\Big| \leq M \Big\},
    \quad \text{and} \quad
    I_j \subseteq \Big\{ i \in \Z \colon \Big|i - \Big\lfloor
    \frac{\beta}{\alpha} \, \frac{1}{c}
    \, j \Big\rfloor\Big| \leq M \Big\},
  \]
  where $c = c(A, B) := \ln |\det A| / \ln |\det B|$.
  In particular, this implies that
  \[
    \sup_{j \in \Z} |I_j| + \sup_{i \in \Z} |J_i| < \infty.
  \]
  Therefore, an application of \Cref{lem:hom_covers} implies that $A^\ast$ and $B^\ast$ are
  equivalent, and hence so are $A$ and $B$ by
  \Cref{cor:adjoint_equivalent}.

  It remains to show that $\alpha = \beta$.
  To see this, note that, for all $j \in \Z$, it holds that
  \begin{align*}
    |\det B|^j
    & \lesssim_{P} \Lebesgue{(B^\ast)^j P}
      \leq \Measure \Big(\bigcup_{i \in I_j} (A^\ast)^i Q \Big)
      \leq \sum_{i \in I_j} \Lebesgue{(A^\ast)^i Q} \\
    & \lesssim_{Q} \sum_{k = -M}^M |\det A|^{\lfloor \frac{\beta}{\alpha} \, \frac{1}{c}
      \, j \rfloor +k}
      \lesssim_{M,A}  |\det A|^{ \frac{\beta}{\alpha} \frac{\ln(|\det B|)}{\ln(|\det A|)} j}
      = |\det B|^{ \frac{\beta}{\alpha}j} .
  \end{align*}
  Since $|\det B| \neq 0$, this is only possible for
  $\frac{\beta}{\alpha} = 1$, and hence $\alpha = \beta$ as claimed.
\end{proof}

\subsection{The case \texorpdfstring{$\alpha = 0$}{α = 0} and \texorpdfstring{$p < \infty$}{p < ∞}}%
\label{sub:NecessityProofAlphaZero}

In this section, we prove the following theorem, showing that if two Triebel-Lizorkin
spaces coincide and the matrices are \emph{not} equivalent, then necessarily $q = 2$.
The only shortcoming of this theorem is that it only applies when $p < \infty$.
We will deal with the case $p = \infty$ in the following subsection.

\begin{theorem}\label{lem:KhintchineArgument}
  Let $A,B \in \GL(d,\R)$ be expansive, $p \in (0,\infty)$ and $q \in (0,\infty]$.
  Suppose that
  \begin{equation}\label{eq:TLone_coincidence}
   \| f \|_{\TLzero(A)} \asymp \| f \|_{\TLzero(B)},
   \quad \text{for all} \quad f \in \mathcal{F}^{-1} (C_c^{\infty} (\mathbb{R}^d \setminus \{0\})).
  \end{equation}
  If $A$ and $B$ are not equivalent, then $q = 2$.

  In particular, if $\TLzero(A) = \TLzero(B)$ and $A$ and $B$ are not equivalent, then $q=2$.
\end{theorem}

The following observation will be key in proving \Cref{lem:KhintchineArgument}.
It provides a condition under which the hypotheses of \Cref{prop:TL-norm-estimate-sequence}
are satisfied for $\TL(A)$.

\begin{lemma}\label{lem:A-B-nequiv-construction}
  Let $A, B \in \GL(d, \mathbb{R})$ be expansive and suppose that $\sup_{j \in \Z} |I_j| = \infty$,
  with $I_j$ as defined in \Cref{eq:SpecializedIntersectionSets}.

  Then, for every $K \in \N$, there exist $\delta_0 > 0$, $j_0 \in \Z$, points
  $\eta_1, \dots \eta_K \in \R^d$, and a (strictly) increasing sequence
  $i_1, \dots, i_K \in \Z$ with $|i_k - i_{k'}| > 2N$ for $k \neq k'$,
  where $N \in \N$ as in \eqref{eq:NeighborControl}, such that the following assertions hold:
  \begin{enumerate}[(i)]
   \item $B_{\delta_0}(\eta_k) \subseteq (A^\ast)^{i_k}Q \cap (B^\ast)^{j_0} P$ for all $k = 1, \dots, K$;
   \item $|\phi(x)| \geq \frac{1}{2}|\phi(0)|$ for all  $x \in \delta_0 A^{-i_1}\Omega_A$.
  \end{enumerate}
  In particular, the assumptions (a) and (b) of
  \Cref{prop:TL-norm-estimate-sequence} are satisfied for $\TL(A)$.
\end{lemma}

\begin{proof}
  Let $K \in \N$ be arbitrary.
  Then, since $\sup_{j \in \Z} |I_j| = \infty$,
  there exists  $j_0 \in \Z$ such that $|I_{j_0}| \geq (2 N + 1) K$.
  Define $\Z_n := n + (2 N + 1) \, \Z$ for $n = 0,\dots,2N$.
  Since $I_{j_0} = \bigcup_{n = 0}^{2N} (I_{j_0} \cap \Z_n)$, there needs to
  be at least one $n_0 \in \{0, \dots, 2N\}$ such that
  $|I_{j_0} \cap \Z_{n_0}| \geq K$.
  Hence, we can choose a strictly increasing sequence $i_1, \dots, i_K \in I_{j_0} \cap \Z_{n_0}$,
  which in particular implies that $|i_k - i_{k'}| \geq 2N+1$ for $k \neq k'$.
  Since $(A^\ast)^{i_k}Q \cap (B^\ast)^{j_0}P \neq \emptyset$ is open for
  all $k=1, \dots, N$, there exist $\eta_1, \dots, \eta_K$ and a constant
  $\delta_1 > 0$ such that
  \begin{equation*}
    B_{\delta_1}(\eta_k) \subseteq (A^\ast)^{i_k}Q \cap (B^\ast)^{j_0} P
    \qquad \text{for all} \quad k=1, \dots, K.
  \end{equation*}
  Finally, continuity of $\phi \in \Schwartz(\R^d)$ implies
  (because of $|\phi(0)| = \phi(0) = \int_{\R^d} \widehat{\phi(\xi)} \, d \xi > 0$)
  the existence of $\delta_2 > 0$ such that
  $|\phi(x)| \geq \frac{1}{2}|\phi(0)|$ for all
  $x \in \delta_2 A^{-i_1}\Omega_A$, which completes the proof by
  setting $\delta_0 := \min\{\delta_1, \delta_2\}$.
\end{proof}

Another key ingredient used in the proof of \Cref{lem:KhintchineArgument} is  Khintchine's inequality,
see, e.g., \cite[Proposition~4.5]{WolffLecturesOnHarmonicAnalysis}.
We include its statement for the convenience of the reader.

\begin{lemma}[Khintchine]\label{lem:KhintchineInequality}
  Let $\theta = (\theta_1, \dots, \theta_K)$ be a random vector with $\theta \sim U(\{ \pm 1 \}^K)$
  (i.e., $\PP (\theta = \eta) = \frac{1}{2^K}$ for every $\eta \in \{ \pm 1 \}^K$).
  For any $p \in (0, \infty)$, denoting the expectation with respect to $\theta$ by $\EE_{\theta}$,
  it holds that
  \[
    \EE_{\theta} \bigg|\sum_{k=1}^K \theta_k \, a_k\bigg|^p \asymp
    \bigg(\sum_{k=1}^{K} |a_k|^2\bigg)^{p/2} \qquad \text{for all} \quad (a_k)_{k = 1}^K \in \CC^K,
  \]
  where the implied constant only depends on $p$.
\end{lemma}

We will now provide the proof of \Cref{lem:KhintchineArgument}.

\begin{proof}[Proof of \Cref{lem:KhintchineArgument}]
  If $A$ and $B$ are not equivalent, then neither are $A^\ast$ and $B^\ast$
  (cf.\ \Cref{cor:adjoint_equivalent}).
  Hence, an application of \Cref{lem:hom_covers}
  implies for  $\big((A^\ast)^i Q\big)_{i \in \Z}$ and
  $\big((B^\ast)^j P\big)_{j \in \Z}$ that
  \begin{equation*}
    \sup_{i \in \Z} |J_i| + \sup_{j \in \Z} |I_j|  = \infty.
  \end{equation*}
  By exchanging the roles of $A$ and $B$ if necessary, it may be assumed that
  $\sup_{j \in \Z} |I_j| = \infty$, so that the assumption of Lemma~\ref{lem:A-B-nequiv-construction} is satisfied.
  Using Lemma \ref{lem:A-B-nequiv-construction}, it will be shown that
  \begin{equation}
    \label{eq:KhintchineProofTargetEstimate}
    \| c \|_{\ell^q}
    \asymp \| c \|_{\ell^2 }
    \qquad \text{for all} \quad \, K \in \N \text{ and } c \in \CC^K,
  \end{equation}
  where the implied constant is independent of $K$ and $c$.
  This easily implies $q = 2$.

  For showing \eqref{eq:KhintchineProofTargetEstimate}, let $K \in \N$ and $c \in \CC^K$ be arbitrary.
  Then an application of \Cref{lem:A-B-nequiv-construction} yields some $j_0 \in \Z$,
  points $\eta_1, \dots, \eta_K \in \R^d$, a strictly increasing sequence $i_1, \dots, i_K \in \Z$,
  and $\delta_0 > 0$ such that $B_{\delta_0}(\eta_k) \subseteq (B^\ast)^{j_0} P$
  for all $k \in \{ 1,\dots,K \}$ and such that the assumptions
  of \Cref{prop:TL-norm-estimate-sequence} are satisfied.
  \Cref{prop:TL-norm-estimate-sequence} thus implies
  for fixed but arbitrary $0 < \delta \leq \delta_0$,
  and any $\theta \in \{ \pm 1 \}^K$ that the function
  \(
    f_{\theta}: = \sum_{k = 1}^K \theta_k \, c_k M_{\eta_k}\phi_\delta
  \)
  satisfies $\| f_\theta \|_{\TLzero (A)} \asymp \delta^{d (1 - \frac{1}{p})} \| c \|_{\ell^q}$.
  On the other hand, it holds $\supp \widehat{f_{\theta}} \subseteq (B^\ast)^{j_0} P$
  for all $0 < \delta \leq \delta_0$, and thus \Cref{prop:TL-norm-estimate-simple-Fourier-support}
  is applicable for $\TLzero(B)$.
  Consequently, \Cref{eq:TLone_coincidence} implies that
  \begin{equation}
    \label{eq:KhintchineArgumentAlmostDone}
    \delta^{d  (1 -1/p)} \| c \|_{\ell^q }
    \asymp \| f_{\theta} \|_{\TLzero(A)}
    \asymp \| f_{\theta} \|_{\TLzero(B)}
    \asymp \| f_{\theta} \|_{L^p}
    \qquad \text{for all} \quad \theta \in \{ \pm 1 \}^K.
  \end{equation}
  Using Khintchine's inequality (\Cref{lem:KhintchineInequality}), we see that if we
  take $\theta \sim U(\{ \pm 1 \}^K)$ as a random vector, then
  \begin{align*}
    \EE_{\theta} \| f_{\theta} \|_{L^p}^p
    & = \EE_{\theta}
        \int_{\R^d}
          \bigg|
            \sum_{k=1}^{K}
              \theta_k \, c_k \, e^{2 \pi i \eta_k \cdot x }
          \bigg|^p
           |\phi_\delta(x)|^p
        \, d x \\
    & = \int_{\R^d}
          |\phi_\delta(x)|^p \,
           \EE_{\theta}
                \bigg|
                  \sum_{k=1}^{K}
                    \theta_k \, c_k \, e^{2 \pi i  \eta_k \cdot x }
                \bigg|^p
        \, d x \\
    & \asymp_p \int_{\R^d}
                 |\phi_\delta(x)|^p
                 \bigg(
                   \sum_{k=1}^{K}
                     |c_k \,e^{2 \pi i  \eta_k \cdot x }|^2
                 \bigg)^{p/2}
               \, d x \\
    & \asymp_{p,\phi} \delta^{d  (p-1)} \| c \|_{\ell^2}^p .
  \end{align*}
  In combination with \eqref{eq:KhintchineArgumentAlmostDone}, this easily implies that
  \Cref{eq:KhintchineProofTargetEstimate} holds.
\end{proof}

The finer analysis in the case where $\alpha = 0$ and $q = 2$ can be performed by
using that $\dot{\mathbf{F}}^{0}_{p,2}(A)$ coincides with the anisotropic Hardy space $H^p (A)$
and using the classification results of \cite[Section 10]{bownik2003anisotropic}.
The details are as follows:

\begin{theorem}\label{thm:PFiniteQEqualTwoCase}
  Let $A, B \in \mathrm{GL}(d, \mathbb{R})$ be expansive and $p \in (0,\infty)$.
  If $\dot{\mathbf{F}}^{0}_{p,2}(A) = \dot{\mathbf{F}}^{0}_{p,2}(B)$,
  then at least one of the following cases holds:
  \begin{enumerate}
   \item[(i)] $A$ and $B$ are equivalent, or
   \item[(ii)] $p \in (1, \infty)$.
  \end{enumerate}
\end{theorem}

\begin{proof}
Let $p \in (0,\infty)$ and denote by $H^p(A)$ the anisotropic Hardy space
introduced in \cite{bownik2003anisotropic}.
By \cite[Theorem 7.1]{bownik2007anisotropic}, it follows that $\dot{\mathbf{F}}^{0}_{p,2}(A) = H^p (A)$.
Hence, if $\dot{\mathbf{F}}^{0}_{p,2}(A) = \dot{\mathbf{F}}^{0}_{p,2}(B)$, then $H^p (A) = H^p (B)$.

If $p \in (0, 1]$, then by \mbox{\cite[Theorem 10.5]{bownik2003anisotropic}}
(see also \cite[Theorem 2.3]{bownik2020pde} for a corrected statement),
the identity $H^p (A) = H^p (B)$ implies that $A$ and $B$ are equivalent.
Thus, (i) holds.
\end{proof}

A combination of \Cref{lem:KhintchineArgument} and \Cref{thm:PFiniteQEqualTwoCase} yields the following:

\begin{corollary}
  Let $A,B \in \GL(d,\R)$ be expansive, $p \in (0,\infty)$ and $q \in (0,\infty]$.
  Suppose that $\TLzero(A) = \TLzero(B)$.
  Then at least one of the following cases holds:
  \begin{enumerate}[(i)]
   \item $A$ and $B$ are equivalent;
   \item $q = 2$ and $p \in (1,\infty)$.
  \end{enumerate}
\end{corollary}

\subsection{The case \texorpdfstring{$\alpha = 0$}{α = 0} and \texorpdfstring{$p = \infty$}{p = ∞}}%
\label{sub:PInftyAlmostProof}

This section provides the following theorem, which finishes the
necessary conditions of \Cref{thm:NecessaryCondition}.

\begin{theorem}\label{lem:PInftyMainArgument}
  Let $A,B \in \GL(d, \R)$ be expansive and $q \in (0, \infty]$.
  If $\TLnaked^{0}_{\infty,q} (A) = \TLnaked^0_{\infty,q}(B)$, then $A$ and $B$ are equivalent.
\end{theorem}

The following lemma will reduce the proof of \Cref{lem:PInftyMainArgument} to the case $q \geq 1$.

\begin{lemma}\label{lem:PInftyBanachReduction}
  Let $A,B \in \GL(d, \R)$ be expansive and $q \in (0, \infty]$.
  If $\TLnaked^{0}_{\infty,q} (A) = \TLnaked^0_{\infty,q}(B)$ and
  the matrices $A$ and $B$ are not equivalent, then $q \geq 1$.
\end{lemma}

\begin{proof}
  The claim is trivial for $q = \infty$; therefore, we can assume that $q < \infty$.
  Since $A$ and $B$ are not equivalent,
  \Cref{cor:adjoint_equivalent} and \Cref{lem:hom_covers} again
  imply for the covers $\big((A^\ast)^i Q\big)_{i \in \Z}$ and
  $\big((B^\ast)^j P\big)_{j \in \Z}$ that
  \begin{equation*}
    \sup_{i \in \Z} |J_i| + \sup_{j \in \Z} |I_j|  = \infty,
  \end{equation*}
  where we may assume $\sup_{j \in \Z} |I_j| = \infty$ by interchanging $A$
  and $B$ if necessary.

  For $K \in \N$ arbitrary, we now invoke
  \Cref{lem:A-B-nequiv-construction} to obtain $j_0 \in \Z$,
  $\eta_1, \dots, \eta_K \in \R^d$, a strictly increasing sequence $i_1, \dots, i_K \in \Z$,
  and some $\delta_0 > 0$ such that the assumptions of \Cref{prop:TL-norm-estimate-sequence}
  are satisfied and such that $B_{\delta_0}(\eta_k) \subseteq (B^\ast)^{j_0} P$
  for all $k \in \{ 1,\dots,K \}$.
  \Cref{prop:TL-norm-estimate-sequence} thus implies for any $0 < \delta \leq \delta_0$
  that each of the functions
  \begin{equation*}
    f_c := \sum_{k=1}^{K} c_k M_{\eta_k} \phi_\delta,
    \qquad c \in \CC^K,
  \end{equation*}
  satisfies $\| f_c \|_{\TLnaked^{0}_{\infty,q} (A)} \asymp \delta^d \, \| c \|_{\ell^q}$.
  Since $\supp \widehat{f_c} \subseteq (B^\ast)^{j_0}P$,
  \Cref{prop:TL-norm-estimate-simple-Fourier-support} is applicable for
  $\TLnaked^{0}_{\infty,q} (B)$.
  Consequently, and recalling \eqref{eq:CoincidenceImpliesNormEquivalence}, we see that
  \begin{equation*}
    \delta^d \, \|c\|_{\ell^q}
    \asymp \| f_c \|_{\TLnaked^0_{\infty,q} (A)}
    \asymp \| f_c \|_{\TLnaked^0_{\infty,q} (B)}
    \asymp \| f_c \|_{L^\infty}
    \leq \| c \|_{\ell^1} \, \| \phi_\delta \|_{L^\infty}
    \lesssim  \delta^d \, \| c \|_{\ell^1},
  \end{equation*}
  which can only hold for $q \geq 1$.
\end{proof}

By duality, we now provide a proof of Theorem~\ref{lem:PInftyMainArgument}.

\begin{proof}[Proof of \Cref{lem:PInftyMainArgument}]
  Arguing by contradiction, we assume that $A$ and $B$ are not equivalent.
  Then \Cref{lem:PInftyBanachReduction} implies that $q \geq 1$.

  First, suppose that $q \in (1, \infty]$, so that its conjugate
  exponent $q'$ satisfies $q' \in [1,\infty)$.
  Then \mbox{\cite[Theorem~4.8]{bownik2008duality}} shows that
  $\TLnaked^0_{\infty,q} (A)$ is the dual space of
  $\TLnaked^0_{1,q'}(A)$ (with equivalent norms).
  Likewise, it follows that $\TLnaked^0_{\infty,q}(B)$ is the dual space of
  $\TLnaked^0_{1,q'}(B)$ (with equivalent norms).
  By the first part of \Cref{lem:CoincidenceImpliesNormEquivalence}, we have for
  \(
    f \in [\TLnaked^0_{1,q'}(A)]'
        = \TLnaked^0_{\infty,q}(A)
        = \TLnaked^0_{\infty,q}(B)
        = [\TLnaked^0_{1,q'}(B)]'
  \)
  that
  \[
    \| f \|_{[\TLnaked^0_{1,q'} (A)]'}
    \asymp \| f \|_{\TLnaked^0_{\infty,q} (A)}
    \asymp \| f \|_{\TLnaked^0_{\infty,q} (B)}
    \asymp \| f \|_{[\TLnaked^0_{1,q'} (B)]'}
    .
  \]
  Therefore, by the usual dual characterization of the norm, it holds that
  \[
    \| g \|_{\TLnaked^0_{1,q'}(A)}
    = \sup_{\substack{f \in [\TLnaked^0_{1,q'}(A)]'\\ \| f \|_{[\TLnaked^0_{1,q'}(A)]'} \leq 1}}
        |\langle f, g \rangle|
    \asymp \sup_{\substack{f \in [\TLnaked^0_{1,q'}(B)]'\\ \| f \|_{[\TLnaked^0_{1,q'}(B)]'} \leq 1}}
           |\langle f, g \rangle|
    = \| g \|_{\TLnaked^0_{1,q'}(B)}, \quad g \in \mathcal{S}_{0} (\mathbb{R}^d).
  \]
  Second, if $q = 1$, then it follows directly from
  \Cref{prop:dualnormF1inf} that
  \[
    \| g \|_{\dot{\mathbf{F}}^{0}_{1,\infty}(A)}
    \asymp \sup_{\substack{ f \in \dot{\mathbf{F}}^0_{\infty, 1}(A) \\
                 \| f \|_{\dot{\mathbf{F}}^0_{\infty, 1}(A) } \leq 1} }
             |\langle f, g \rangle|
    \asymp \sup_{\substack{ f \in \dot{\mathbf{F}}^0_{\infty, 1}(B) \\
                 \| f \|_{\dot{\mathbf{F}}^0_{\infty, 1}(B) } \leq 1} }
             |\langle f, g \rangle|
    \asymp \| g \|_{\dot{\mathbf{F}}^{0}_{1,\infty}(B)},
    \quad g \in \mathcal{S}_{0} (\mathbb{R}^d).
  \]

  In combination, for any $q \in [1, \infty]$, this yields
  $ \| g \|_{\dot{\mathbf{F}}^{0}_{1,q'}(A)} \asymp \| g \|_{\dot{\mathbf{F}}^{0}_{1,q'}(B)}$ for all
  $g \in \mathcal{S}_{0} (\mathbb{R}^d)$.
  Since $A$ and $B$ are not equivalent, an application of \Cref{lem:KhintchineArgument}
  shows that $q' = 2$ and hence $q = 2$.
  But for $p = 1$, $q = 2$, the above norm equivalence holds on a common dense subset,
  hence $\dot{\mathbf{F}}^{0}_{1,2}(A) = \dot{\mathbf{F}}^{0}_{1,2}(B)$ by
  the second part of \Cref{lem:CoincidenceImpliesNormEquivalence}.
  Now \Cref{thm:PFiniteQEqualTwoCase} implies that $A$ and $B$ need
  to be equivalent, a contradiction.
\end{proof}

\section{Sufficient conditions}
\label{sec:sufficient}

This section is devoted to the sufficient conditions of \Cref{thm:intro2}
and consists of the proof of the following theorem.
A key ingredient used in the proof is the maximal characterization of Triebel-Lizorkin spaces
(see \Cref{thm:maximal_characterizations}).

\begin{theorem}
  Let $A, B \in \GL(d,\R)$ be two expansive matrices.
  If $A$ and $B$ are equivalent, then $\TL(A) = \TL(B)$ for all $p,q \in (0,\infty]$
  and $\alpha \in \R$.
\end{theorem}

\begin{proof}
  Let $A, B \in \GL(d,\R)$ be two equivalent expansive matrices.
  Suppose $\varphi, \psi \in \Schwartz(\R^d)$ are analyzing vectors
  for $A$ respectively $B$ satisfying additionally \Cref{eq:analyzing3}, i.e., so that
  $Q:= \bigl\{ \xi \in \R^d \colon \widehat{\varphi} (\xi) \neq 0 \bigr\}$ and
  $P := \bigl\{ \xi \in \R^d \colon \widehat{\psi}(\xi) \neq 0 \bigr\}$ are relatively compact in
  $\mathbb{R}^d \setminus \{0\}$, and
  \[
    \sum_{i \in \Z} \widehat{\varphi} \bigl( (A^*)^{-i} \xi \bigr) = 1
    =  \sum_{j \in \Z} \widehat{\psi} \bigl( (B^*)^{-j} \xi \bigr)
    \quad \text{for all} \quad \xi \in \mathbb{R}^d \setminus \{0\}.
  \]
  Then $\bigl( (A^*)^i Q \bigr)_{i \in \Z}$ and $\bigl( (B^*)^j P \bigr)_{j \in \Z}$ are covers
  of $\mathbb{R}^d \setminus \{0\}$.
  Furthermore, a straightforward calculation
  yields $\widehat{\varphi_i} = \widehat{\varphi} ((A^*)^{-i} \cdot)$ and
  $\widehat{\psi_j} = \widehat{\psi} ((B^*)^{-j} \cdot)$, and
  hence $\widehat{\varphi_i} \equiv 0$ outside of $(A^*)^i Q$ and
  $\widehat{\psi_j} \equiv 0$ outside of $(B^*)^j P$.
  Since $A$ and $B$ are equivalent, so are $A^*$ and $B^*$
  (cf.\ Corollary~\ref{cor:adjoint_equivalent}.)

  For fixed $i \in \Z$, define $\Psi_i \in \Schwartz(\R^d)$ as
  \begin{equation*}
    \Psi_i := \sum_{j \in J_i} \psi_j,
  \end{equation*}
  where $J_i := \{ j \in \Z: (A^*)^i Q \cap (B^*)^j P \neq \emptyset \}$
  is finite by \Cref{lem:hom_covers}.
  Clearly, $\widehat{\Psi_i} = \sum_{j \in J_i} \widehat{\psi_j} \equiv 1$ on
  $(A^*)^i Q \supseteq \bigl\{ \xi \in \R^d \colon \widehat{\varphi_i} (\xi) \neq 0 \bigr\}$
  by construction.
  Therefore,
  \begin{equation}
    \label{eq:ConvolutionIdentity}
    \varphi_i \ast \Psi_i = \varphi_i \quad \text{for all} \quad i \in \Z.
  \end{equation}
  We will use \eqref{eq:ConvolutionIdentity} to obtain a pointwise estimate
  of the convolution products $f \ast \varphi_i$, $i \in \mathbb{Z}$, in terms of the
  Peetre-type maximal function $\DoubleStar[\psi]{j}f$ for a fixed $\beta > \max \{ 1/p, 1/q \}$, defined by
  \begin{equation*}
    \DoubleStar[\psi]{j}f(x) = \sup_{z \in \R^d} \frac{|(f \ast
      \psi_j)(x+z)|}{(1+ \rho_B(B^jz))^\beta} \quad \text{for all} \quad x \in \R^d;
  \end{equation*}
  see Section \ref{sec:maximal}.
  For fixed $x \in \mathbb{R}^d$, a direct calculation gives
  \begin{align*}
    |(f \ast \varphi_i)(x)|
    &  \leq \sum_{j \in J_i} |(f \ast \psi_j\ast \varphi_i)(x)| \\
    &\leq \sum_{j \in J_i} \int_{\R^d}\frac{|(f \ast \psi_j)(x+y)|}{(1+ \rho_B(B^jy))^\beta}
      \cdot (1+ \rho_B(B^jy))^\beta  \, |\varphi_i(-y)| \, dy \\
    &\leq \sum_{j \in J_i} \DoubleStar[\psi]{j}f(x)
      \int_{\R^d} (1+ \rho_B(B^jy))^\beta  \, |\varphi_i(-y)| \, dy \\
    &= \sum_{j \in J_i} \DoubleStar[\psi]{j}f(x)
      \int_{\R^d} (1+ \rho_B(B^jy))^\beta  \, |\det A|^i \, |\varphi(-A^iy)| \, dy \\
    &= \sum_{j \in J_i} \DoubleStar[\psi]{j}f(x)
      \int_{\R^d} (1+ \rho_B(B^jA^{-i}z))^\beta \, |\varphi(-z)| \, dz.
          \numberthis \label{eq:ConvDecomp}
  \end{align*}
  To bound the integral in \eqref{eq:ConvDecomp}, we note that, since
  $\rho_A,\rho_B$ are equivalent, we have
  \begin{equation*}
    \rho_B(B^j A^{-i} z)
    = |\det B|^j \rho_B(A^{-i}z)
    \leq C |\det B|^j \rho_A(A^{-i}z)
    = C |\det B|^j |\det A|^{-i} \rho_A(z).
  \end{equation*}
  Lemma~\ref{lem:detQuotient} implies that
  $ |\det A|^i \asymp |\det B|^j$ for $j \in J_i$ with implicit
  constants independent of $i \in \Z, j \in J_i$.
  Consequently, $(1+ \rho_B(B^jA^{-i}z))^\beta \lesssim (1+ \rho_A(z))^\beta$ for
  all $z \in \mathbb{R}^d$.
  Since $\varphi \in \mathcal{S}(\mathbb{R}^d)$, it follows that for every
  $N \in \mathbb{N}$, there exists $C_N > 0$ such that
  $|\varphi(z)| \leq C_N (1+\rho_A (z))^{-N}$, see, e.g., \cite[Section 3]{bownik2003anisotropic}.
  Combining these observations with \eqref{eq:ExpansiveConsequence} easily yields
  \[
   \int_{\R^d} (1+ \rho_B(B^jA^{-i}z))^\beta \, |\varphi(-z)| \, dz \lesssim 1
  \]
  with implicit constant independent of $i \in \Z$ and  $j \in J_i$.
  Using $|\det A|^i \asymp |\det B|^j$ for $j \in J_i$ once again, it follows thus that
  \begin{equation}
    \label{eq:A-B-swap}
    |\det A|^{\alpha i} |(f \ast \varphi_i)(x)|
    \lesssim  \sum_{j \in J_i} |\det B|^{\alpha j} \DoubleStar[\psi]{j}f(x)
    \quad \text{for all} \quad x \in \R^d,
  \end{equation}
  for all $i \in \Z$.

  The remainder of the proof is split into three cases dealing with
  $p < \infty$, $p = \infty$ and $q < \infty$, and $p = q = \infty$ separately.
\\~\\
  \textbf{Case 1:} $p \in (0, \infty)$.
  We only prove this case for $q \in (0, \infty)$, since analogous arguments using suprema yield
  the case for $q = \infty$.
  Hence, for $q < \infty$, raising \eqref{eq:A-B-swap} to the $q$-th power
  and summing over $i \in \Z$ results in
  \begin{align*}
    \sum_{i \in \Z} ( |\det A|^{\alpha i} |(f \ast \varphi_i)(x)|)^q
    &\lesssim \sum_{i \in \Z}  \Big(\sum_{j \in J_i} |\det B|^{\alpha j}
    \DoubleStar[\psi]{j}f(x)\Big)^q \\
    &\lesssim \sum_{i \in \Z} \sum_{j \in J_i} (|\det B|^{\alpha j}
    \DoubleStar[\psi]{j}f(x))^q,
  \end{align*}
  where we used in the last step that $\sup_{i \in \Z}|J_i| < \infty$ by \Cref{lem:hom_covers}.
  Since \Cref{lem:hom_covers} also implies $\sup_{j \in \Z}|I_j| < \infty$ for
  $I_j:= \{ i \in \Z: (A^*)^iQ \cap (B^*)^jP \neq \emptyset \}$, it follows that
  \begin{align*}
    \sum_{i \in \Z} \sum_{j \in J_i} (|\det B|^{\alpha j}
    \DoubleStar[\psi]{j}f(x))^q
    &= \sum_{j \in \Z} \sum_{i \in I_j} (|\det B|^{\alpha j}
    \DoubleStar[\psi]{j}f(x))^q \\
    &\lesssim \sum_{j \in \Z} (|\det B|^{\alpha j}
    \DoubleStar[\psi]{j}f(x))^q.
  \end{align*}
  Consequently, we have
  \begin{equation*}
    \|f\|_{\TL(A)}
    = \bigg\|
        \Big(
          \sum_{i \in \Z}
            ( |\det A|^{\alpha i} \, |f \ast \varphi_i|)^q
        \Big)^{1/q}
      \bigg\|_{L^p}
    \lesssim \bigg\|
               \Big(
                 \sum_{j \in \Z}
                   (|\det B|^{\alpha j} \, \DoubleStar[\psi]{j} f)^q
                \Big)^{1/q}
             \bigg\|_{L^p}
    \asymp \|f\|_{\TL(B)},
  \end{equation*}
  where the last equivalence follows from Theorem~\ref{thm:maximal_characterizations}.
  Exchanging the roles of $A$ and $B$ yields the converse inequality and therefore
  $\TL(A) = \TL(B)$ in this case.
\\~\\
  \textbf{Case 2:} $p = \infty$, $q \in (0, \infty)$.
  Let $\ell \in \Z$ be arbitrary.
  Again, we raise \eqref{eq:A-B-swap} to the $q$-th power, sum over $i \geq -\ell$,
  and use the fact that $\sup_{i \in \Z}|J_i| < \infty$.
  This gives
  \begin{align*}
    \sum_{i = - \ell}^{\infty}
      ( |\det A|^{\alpha i} |(f \ast \varphi_i)(x)|)^q
    & \lesssim \sum_{i = - \ell}^{\infty}
               \bigg(
                 \sum_{j \in J_i}
                  |\det B|^{\alpha j} \, \DoubleStar[\psi]{j}f(x)
               \bigg)^q \\
    &\lesssim \sum_{i = - \ell}^{\infty} \,\,
                \sum_{j \in J_i}
                  (|\det B|^{\alpha j} \, \DoubleStar[\psi]{j}f(x))^q.
  \end{align*}
  \Cref{lem:IjJiInclusions} yields the existence of $N_1 \in \N$ such that
  $J_i \subseteq \{j \in \Z: |j - \lfloor c i \rfloor| \leq N_1  \}$ for all $i \in \Z$, where
  $c = c(A,B) := \ln |\det A| / \ln |\det B|$.
  Hence, $j \geq \lfloor - c \ell \rfloor - N_1$ for all $j \in \bigcup_{i = - \ell}^\infty J_i$.
  By setting $\ell_1:= \lfloor c \ell \rfloor + N_1 + 1 \geq -(\lfloor - c \ell \rfloor - N_1)$,
  we thus obtain
  \begin{align*}
    \sum_{i = - \ell}^{\infty} \,\,
      \sum_{j \in J_i}
        (|\det B|^{\alpha j} \, \DoubleStar[\psi]{j}f(x))^q
    & \leq \sum_{j = \lfloor- c \ell \rfloor -N_1}^{\infty} \,\,
             \sum_{i \in I_j}
               (|\det B|^{\alpha j} \, \DoubleStar[\psi]{j}f(x))^q \\
    &\lesssim \sum_{j = - \ell_1}^{\infty} (|\det B|^{\alpha j}
    \DoubleStar[\psi]{j}f(x))^q,
  \end{align*}
  where in the last step we used that $\sup_{j \in \Z}|I_j| < \infty$ for
  $I_j:= \{ i \in \Z: (A^*)^iQ \cap (B^*)^jP \neq \emptyset \}$.
  In combination, the above two estimates show that, for any $\ell \in \Z$,
  \begin{equation}
    \label{eq:case2-part1}
     \sum_{i = - \ell}^{\infty} ( |\det A|^{\alpha i} |(f \ast \varphi_i)(x)|)^q
    \lesssim \sum_{j = - \ell_1}^{\infty} (|\det B|^{\alpha j}
    \DoubleStar[\psi]{j}f(x))^q \quad \text{for all} \quad x \in \R^d.
  \end{equation}

  Let $\Omega_A, \Omega_B \subseteq \R^d$ be the fixed ellipsoids used
  in the definition of $\rho_A$ resp.\ $\rho_B$ (cf.\ Section~\ref{sec:expansive}).
  Then $A^\ell \Omega_A = \{ x \in \R^d : \rho_A(x) < |\det A|^\ell\}$,
  and thus any $x \in A^\ell \Omega_A$ satisfies
  \begin{equation*}
    \rho_B(x) \leq C \rho_A(x) < C |\det A|^\ell = C |\det B|^{c \ell}
    \leq |\det B|^{\lfloor c \ell \rfloor +N_2}
  \end{equation*}
  with
  \(
    N_2
    := \max \big\{ 1, \, \lceil \ln C / \ln |\det B| \rceil \big\} + N_1
    \geq \lceil \ln C / \ln |\det B| \rceil + 1
    .
  \)
  Consequently, we have for all $\ell \in \Z$ the inclusion
  \begin{equation}\label{eq:case2-part2}
    A^\ell \Omega_A \subseteq B^{\lfloor c \ell \rfloor + N_2} \Omega_B = B^{\ell_2} \Omega_B,
    \qquad \text{where} \qquad
    \ell_2 := \lfloor c \ell \rfloor + N_2
    .
  \end{equation}
  Now let $w \in \R^d$ also be arbitrary.
  Then \eqref{eq:case2-part1} and \eqref{eq:case2-part2} yield
  \begin{align*}
    &\frac{1}{|\det A|^\ell}\int_{A^\ell \Omega_A +w}
    \sum_{i = - \ell}^{\infty} ( |\det A|^{\alpha i} |(f \ast \varphi_i)(x)|)^q \; dx \\
    & \quad \quad \quad \quad
      \lesssim
      \frac{1}{|\det A|^\ell}
      \int_{B^{\ell_2} \Omega_B + w}
        \sum_{j = - \ell_1}^{\infty}
          (|\det B|^{\alpha j} \DoubleStar[\psi]{j}f(x))^q
    \;  dx.
  \end{align*}
  Note that $N_1 + 1 \leq N_2$ and hence $\ell_1 \leq \ell_2$.
  Therefore, we obtain
  \begin{align*}
    & \frac{1}{|\det A|^\ell}
      \int_{A^\ell \Omega_A +w}
        \sum_{i = - \ell}^{\infty}
          ( |\det A|^{\alpha i} \, |(f \ast \varphi_i)(x)|)^q
     \; dx \\
    & \quad \quad \quad \quad
      \lesssim \frac{1}{|\det A|^\ell}
               \int_{B^{\ell_2} \Omega_B + w}
                 \sum_{j = - \ell_2}^{\infty}
                   (|\det B|^{\alpha j} \, \DoubleStar[\psi]{j}f(x))^q
              \; dx \\
    &  \quad \quad \quad \quad
       \lesssim \frac{1}{|\det B|^{\ell_2}}
                \int_{B^{\ell_2} \Omega_B + w}
                  \sum_{j = - \ell_2}^{\infty}
                    (|\det B|^{\alpha j} \, \DoubleStar[\psi]{j}f(x))^q
               \; d x,
      \numberthis \label{eq:case2-part3}
  \end{align*}
  where we used in the last step that
  \(
    |\det A|^{\ell}
    = |\det B|^{c\ell}
    \geq |\det B|^{\lfloor c \ell \rfloor}
    \gtrsim |\det B|^{\ell_2}.
  \)
  Taking the $q$-th root and the supremum over $\ell_2, \ell \in \Z$ and $w \in \R^d$ yields
  \begin{equation*}
    \|f\|_{\TLi(A)}
    \lesssim \sup_{\ell_2 \in \Z, w \in \R^d}
             \bigg(
               \frac{1}{|\det B|^{\ell_2}}
               \int_{B^{\ell_2} \Omega_B + w}
                 \sum_{j = - \ell_2}^{\infty}
                   (|\det B|^{\alpha j} \, \DoubleStar[\psi]{j}f(x))^q
              \; d x
             \bigg)^{1/q}
    \asymp \|f\|_{\TLi(B)},
  \end{equation*}
  where the last equivalence follows again from the maximal
  characterizations of Theorem~\ref{thm:maximal_characterizations}.
  Exchanging the roles of $A$ and $B$ yield the converse norm
  estimate, and therefore it yields that $\TLi(A) = \TLi(B)$.
\\~\\
  \textbf{Case 3:} $p = q= \infty$.  By \Cref{eq:A-B-swap}, it follows that
  \begin{align*}
   \| | \det A|^{\alpha i} (f \ast \varphi_i) \|_{L^{\infty}}
   \lesssim \sum_{j \in J_i} \| |\det B|^{\alpha j} \psi^{**}_{j, \beta} f \|_{L^{\infty}}
   \leq \sum_{j \in J_i} \| |\det B|^{\alpha j} (f \ast \psi_j) \|_{L^{\infty}}
  \end{align*}
for $i \in \Z$. Combining this with \Cref{eq:Bii} yields
\begin{align*}
 \| f \|_{\TLii(A; \varphi)} &\asymp \sup_{i \in \Z} |\det A|^{\alpha i} \| f \ast \varphi_i \|_{L^{\infty}} \\
 &\lesssim \sup_{i \in \Z} \sum_{j \in J_i} |\det B|^{\alpha j} \|  f \ast \psi_j \|_{L^{\infty}} \\
 &\lesssim \sup_{i \in \Z} \sup_{j \in J_i} |\det B|^{\alpha j} \|  f \ast \psi_j \|_{L^{\infty}} \\
 &\leq \sup_{j \in \Z}  |\det B|^{\alpha j} \|  f \ast \psi_j \|_{L^{\infty}} \\
 &\asymp \| f \|_{\TLii(B; \psi)},
\end{align*}
where it is used that $\sup_{i \in \Z} |J_i| + \sup_{j \in \Z} |I_j| < \infty$ by \Cref{lem:hom_covers}.
Exchanging the role of $A$ and $B$ yields $\| \cdot \|_{\TLii(A)} \asymp \| \cdot \|_{\TLii(B)}$, and completes the proof.
\end{proof}

\appendix

\section{Miscellaneous results}

This section contains two results used in the proofs of the main theorems.
The most important such result
is the following convolution relation, which is \cite[Proposition in Section~1.5.1]{triebel2010theory}
with the implied constant written out explicitly.
A proof can also be found in \cite[Theorem~3.4]{VoigtlaenderEmbeddingsOfDecompositionSpaces}.

\begin{proposition}[\cite{triebel2010theory}]\label{prop:GeneralConvolutionRelation}
Let $K_1, K_2 \subseteq \mathbb{R}^d$ be compact and $p \in (0,1]$.
If $f, \psi \in \Schwartz(\R^d)$ satisfy $\supp \widehat{\psi} \subseteq K_1$
and $\supp \widehat{f} \subseteq K_2$, then the following quasi-norm estimate holds:
\[
  \| f \ast \psi \|_{L^p}
  \leq [\Lebesgue{K_1 - K_2}]^{\frac{1}{p} -1}  \| f \|_{L^p}  \| \psi \|_{L^p}
  ,
\]
where $K_1 - K_2 := \{u - v : u \in K_1, v \in K_2 \}$.
\end{proposition}

\begin{corollary}\label{cor:ConvenientConvolutionRelation}
  Let $A \in \GL(d, \R)$ be expansive, let $K \subseteq \R^d$ be compact,
  and let $N \in \N$ and $p \in (0,1)$.
  Then there exists a constant $C = C(A,K,N,p) > 0$ with the following property:

  If $f,g \in \Schwartz(\R^d)$
  satisfy $\supp \widehat{f}, \supp \widehat{g} \subseteq \bigcup_{\ell=-N}^N (A^\ast)^{i + \ell} K$
  for some $i \in \mathbb{Z}$,
  then
  \[
    \| f \ast g \|_{L^p}
    \leq C  |\det A|^{i  (\frac{1}{p} - 1)}  \| f \|_{L^p}  \| g \|_{L^p}
    .
  \]
\end{corollary}

\begin{proof}
  By compactness of $K \subseteq \R^d$, there exists $R = R(A,K,N) > 0$ such that
  \[ \bigcup_{\ell=-N}^N (A^\ast)^{\ell} K \subseteq \overline{B}_R(0).\]
  Setting
  $K_1 := K_2 := (A^\ast)^i \overline{B}_R (0)$, it follows that
  $\supp \widehat{f} \subseteq K_1$, $\supp \widehat{g} \subseteq K_2$, and
  \[
    \Lebesgue{K_1 - K_2}
    \leq \Lebesgue{(A^\ast)^i \overline{B}_{2R}(0)}
    =    |\det A|^i \cdot \Lebesgue{\overline{B}_{2R}(0)}
    .
  \]
  Hence, an application of \Cref{prop:GeneralConvolutionRelation} easily
  yields the claim.
\end{proof}

The second result is the following technical estimate.

\begin{lemma}\label{lem:ImprovedConvolutionBound}
  Let $A \in \GL(d,\R)$ be expansive, let $M > 0$, $N \in \N$, and
  $Q \subseteq \R^d$ be bounded.
  Further, let $\varphi,\phi$ as in \Cref{sub:NecessityProofNotation}.
  Then there exists a constant $C = C(d,M,N,Q,\phi,\varphi,A) > 0$ with the following property:

  If $i,\ell \in \Z$ and $\delta > 0$ are such that $|i - \ell| \leq N$
  and $B_\delta (\eta) \subseteq (A^\ast)^\ell Q$ for some
  $\eta \in \R^d$, then
  \[
    \big( |\phi_\delta| \ast |\varphi_i| \big) (x)
    \leq C  \delta^d  (1 + |\delta x|)^{-M}
  \]
  holds for all $x \in \R^d$.
\end{lemma}

\begin{proof}
  Let $R = R(Q) > 0$ be such that $Q \subseteq B_R (0)$.
  Then
  \[
    B_\delta (0)
    \subseteq B_\delta (\eta) - B_\delta (\eta)
    \subseteq (A^\ast)^\ell Q - (A^\ast)^\ell Q
    \subseteq (A^\ast)^\ell B_{2 R} (0),
  \]
  and thus $(A^\ast)^{-\ell} B_1(0) \subseteq B_{2 R / \delta} (0)$,
  so that $\| (A^\ast)^{-\ell} \| \leq 2 R / \delta$.
  Therefore,
  \[
    \| A^{-i} \|
    = \| (A^\ast)^{-i} \|
    = \| (A^\ast)^{-\ell} (A^\ast)^{\ell - i} \|
    \lesssim_{A,N} \| (A^\ast)^{-\ell} \|
    \leq 2 R / \delta
    ,
  \]
  where it is used that $|i - \ell| \leq N$.
  Thus, given any $y \in \R^d$, it follows that $|\delta y| \leq C_1 \, |A^i y|$
  for a certain constant $C_1 =  C_1(A,N,Q) \geq 1$.
  This implies, for arbitrary $x,y \in \R^d$, that
  \begin{align*}
    1 + |\delta x|
    \leq (1 + |\delta  (x-y)|)  (C_1 + |\delta y|)
    \leq C_1  (1 + |\delta  (x-y)|)  (1 + |A^i y|).
  \end{align*}
  By rearranging, this shows
  $(1 + |\delta  (x-y)|)^{-M} \leq C_1^M  (1 + |\delta x|)^{-M}  (1 + |A^i y|)^M$
  for all $ x, y \in \R^d$.

  Next, since $\phi,\varphi \in \Schwartz(\R^d)$, there exists
  $C_2 = C_2(\phi,\varphi,M,d) > 0$ such that
  \[
    |\phi(x)| \leq C_2  (1 + |x|)^{-M}
    \qquad \text{and} \qquad
    |\varphi(x)| \leq C_2  (1 + |x|)^{-(M + d + 1)}
  \]
  for all $x \in \mathbb{R}^d$.
  Hence,
  \begin{align*}
    \big( |\phi_\delta| \ast |\varphi_i| \big) (x)
    & \leq \delta^d \, |\det A|^i
           \int_{\R^d}
             |\phi(\delta  (x-y))|
              |\varphi(A^i y)|
           \, d y \\
    & \leq C_2^2 \, \delta^d \, |\det A^i|
           \int_{\R^d}
             (1 + |\delta  (x-y)|)^{-M}
              (1 + |A^i y|)^{-(M + d + 1)}
           \, d y \\
    & \leq C_1^M C_2^2 \,
           \delta^d \,
           (1 + |\delta x|)^{-M}
           \int_{\R^d}
             |\det A^i|
              (1 + |A^i y|)^{-(d+1)}
           \, d y \\
    & =    C_1^M \, C_2^2 \, \delta^d
            (1 + |\delta x|)^{-M} \,
           \int_{\R^d} (1 + |z|)^{-(d+1)} \, d z
    .
  \end{align*}
  This easily implies the claim of the lemma.
\end{proof}

\section{Equivalent norm for \texorpdfstring{$\dot{\mathbf{F}}^{0}_{1,\infty}(A)$}{the Triebel-Lizorkin space with p = 1, q = ∞, and α = 0}.}

This section provides a dual characterization for the norm of
$\dot{\mathbf{F}}^{0}_{1,\infty}(A)$, which is used in the proof of
\Cref{lem:PInftyMainArgument}.
Its proof hinges on associated Triebel-Lizorkin sequence spaces
for which we recall the basic objects first.

Let $A \in \mathrm{GL}(d, \mathbb{R})$ be an expansive matrix and let
$\mathcal{D}_A$ be the collection of all \emph{dilated cubes}
\[
 \mathcal{D}_A
 = \big\{ D = A^i ([0,1]^d + k) : i \in \mathbb{Z}, k \in \mathbb{Z}^d \big\}
\]
associated to $A$.
The \emph{scale} of a dilated cube $D = A^i ([0,1]^d + k) \in \mathcal{D}_A$
is defined as $\scale (D) = i$; alternatively, $\scale(D) = \log_{|\det A|} \Measure(D)$.
The \emph{tent} over $D \in \mathcal{D}_A$ is defined as
\[
  \mathcal{T}(D)
  := \big\{
       D' \in \mathcal{D}_A
       :
       \Measure (D' \cap D) > 0
       \quad \text{and} \quad
       \scale(D') \leq \scale(D)
    \big\}.
\]
The following lemma provides a convenient cover for the union of elements of a tent
and will be used in two proofs below.

\begin{lemma}\label{lem:tent-superset}
  There exists $N = N(A,d) \in \N$ such that for all $D \in \mathcal{D}_A$, we have
  \[
    \bigcup_{D' \in \mathcal{T}(D)} D' \subseteq
    \bigcup_{\substack{n \in \Z^d \\ |n| \leq N}} (D + A^{\scale(D)}n).
  \]
\end{lemma}

\begin{proof}
First, let $D' = A^i([0,1]^d+k) \in \mathcal{D}_A$ with $i \leq 0$.
Then
\begin{align}\label{eq:diameterD'}
  \diam (D')
  := \max_{z_1, z_2 \in D'}
       |z_1 - z_2|
   = \max_{x_1, x_2 \in [0,1]^d}
       |A^i(x_1 - x_2)|
  \leq C \lambda_-^i \sqrt{d},
\end{align}
where the inequality used that $| A^i x | \leq C \lambda_-^i | x|$ for
all $x \in \mathbb{R}^d$, see, e.g., \cite[Equations (2.1) and (2.2)]{bownik2003anisotropic}.
Since $\lambda_- > 1$, we can choose $R > 0$ such that $R > C \lambda_-^i \sqrt{d}$
for all $i \leq 0$.
Then, for arbitrary $D' \in \mathcal{T}([0,1]^d)$, it follows that
$D' \cap [0,1]^d \neq \emptyset$, and hence $\dist (x, [0,1]^d) < R$
for all $x \in D'$, so that $D' \subseteq [0,1]^d + B_R (0)$.
Therefore,
\begin{equation}\label{eq:tent-superset-cube}
  \bigcup_{D' \in \mathcal{T}([0,1]^d)} D'
  \subseteq [0,1]^d + B_R (0)
  \subseteq \bigcup_{\substack{n \in \Z^d \\ |n| \leq N}} ([0,1]^d + n)
\end{equation}
for some $N = N(A,d) > 0$.

Second, if $D' = A^i ([0,1]^d + k) \in \mathcal{T}([0,1]^d + \ell)$
for some $\ell \in \mathbb{Z}^d$, then
$D' \cap ([0,1]^d + \ell) \neq \emptyset$ implies that
$\dist(x, [0,1]^d + \ell) \leq \diam(D') < R$ for all $x \in D'$ by
the arguments following \eqref{eq:diameterD'}.
Therefore, by \Cref{eq:tent-superset-cube},
\begin{align}\label{eq:tent-superset-cube2}
  \bigcup_{D' \in \mathcal{T}([0,1]^d + \ell)} D'
  \subseteq [0,1]^d + B_R (0) + \ell
  \subseteq \bigcup_{\substack{n \in \Z^d \\ |n| \leq N}} ([0,1]^d + \ell + n).
\end{align}

At last, let $D = A^j ([0,1]^d + \ell) \in \mathcal{D}_A$ be arbitrary.
Then $D' = A^i([0,1]^d + k) \in \mathcal{T}(D)$ means
$\Measure (D' \cap D) > 0$ and $i \leq j$ by definition of $\mathcal{T}(D)$.
This is clearly equivalent to
\[
  |\det A|^j \Measure \bigl( A^{i - j} ([0,1]^d + k) \cap [0,1^d] + \ell \bigr)
  = \Measure\bigl( A^j (A^{i-j} [0,1]^d + k) \cap A^j([0,1]^d + \ell)\bigr)
  > 0
\]
and $i - j \leq 0$.
Thus, $D' = A^i ([0,1]^d + k) \in \mathcal{T}(D)$ if and only if
\[
  A^{-j} D'
  = A^{i-j} ([0,1]^d + k)
  \in \mathcal{T}([0,1]^d + \ell)
  .
\]
Using \Cref{eq:tent-superset-cube2}, it follows therefore that
\[
  \bigcup_{D' \in \mathcal{T}(D)} D'
  \subseteq A^j
            \bigg(
              \bigcup_{\substack{n \in \Z^d \\ |n| \leq N}}
                ([0,1]^d + \ell + n)
            \bigg)
  = \bigcup_{\substack{n \in \Z^d \\ |n| \leq N}}
      (D + A^j n),
\]
as required.
\end{proof}

The Triebel-Lizorkin sequence spaces $\dot{\mathbf{f}}^0_{1, \infty} (A)$ and
$\dot{\mathbf{f}}^0_{\infty, 1}(A)$ are defined as the collections of
all complex-valued sequences $c = (c_D)_{D \in \mathcal{D}_A}$ satisfying
\[
  \| c \|_{\dot{\mathbf{f}}^0_{1, \infty} (A)}
  := \int_{\mathbb{R}^d}
       \sup_{D \in \mathcal{D}_A}
         \Measure(D)^{-1/2} |c_D| \mathds{1}_D (x)
     \; dx
   < \infty
\]
and
\begin{align} \label{eq:def_fi1}
  \| c \|_{\dot{\mathbf{f}}^0_{\infty, 1}(A)}
  := \sup_{D' \in \mathcal{D}_A}
       \frac{1}{\Measure(D')}
       \int_{D'}
         \sum_{\substack { D \in \mathcal{D}_A \\ \scale(D) \leq \scale(D')}}
           \Measure(D)^{-1/2} |c_D| \mathds{1}_D (x)
       \; dx
   < \infty,
\end{align}
respectively.

The following simple characterization of $\mathbf{f}^0_{\infty, 1} (A)$ will be used below.
This equivalence is already claimed in \cite[Remark 3.5]{bownik2007anisotropic}, but a
short proof is included for the sake of completeness.

\begin{lemma}\label{lem:fi1_char_tent}
  For all complex-valued sequences $c = (c_D)_{D \in \mathcal{D}_A}$,
  \begin{align} \label{eq:fi1_char_tent}
   \| c \|_{\dot{\mathbf{f}}^0_{\infty, 1}(A)} \asymp
   \sup_{D' \in \mathcal{D}_A}
     \frac{1}{\Measure(D')}
     \sum_{D \in \mathcal{T}(D')}
       \Measure(D)^{1/2} |c_D|,
  \end{align}
  where $\mathcal{T}(D')$ denotes the tent over $D' \in \mathcal{D}_A$.
\end{lemma}

\begin{proof}
First, note that interchanging the sum and integral in \Cref{eq:def_fi1} yields that
\begin{equation}
 \| c \|_{\dot{\mathbf{f}}^0_{\infty, 1}(A)}
 = \sup_{D' \in \mathcal{D}_A}
     \frac{1}{\Measure(D')}
     \sum_{\substack { D \in \mathcal{D}_A \\ \scale(D) \leq \scale(D')}}
       \Measure(D)^{-1/2} |c_D| \, \Measure(D \cap D'),
  \label{eq:fi1_char_tent_step1}
\end{equation}
which easily implies the claimed inequality $\lesssim$ in \Cref{eq:fi1_char_tent}.

For the reverse inequality, let $D' = A^j ([0,1]^d + \ell) \in \mathcal{D}_A$ be arbitrary.
Then an application of \Cref{lem:tent-superset} yields $N = N(A,d) \in \mathbb{N}$ such that
\begin{align*}
  T_{D'}
  &:= \frac{1}{\Measure(D')}
      \sum_{D \in \mathcal{T}(D')}
         |c_D| \Measure(D)^{-1/2} \Measure(D) \\
  &= \frac{1}{\Measure(D')}
     \sum_{D \in \mathcal{T}(D')}
       |c_D| \Measure(D)^{-1/2}
       \sum_{\substack{n \in \Z^d \\ |n| \leq N}}
         \Measure\bigl(D \cap A^j ([0,1]^d + \ell + n)\bigr) \\
  &\leq \frac{1}{\Measure(D')}
        \sum_{\substack { D \in \mathcal{D}_A \\ \scale(D) \leq j}} \,\,
          \sum_{\substack{n \in \Z^d \\ |n| \leq N}}
            |c_D| \Measure(D)^{-1/2}  \Measure\bigl(D \cap A^j ([0,1]^d + \ell + n)\bigr) \\
  &= \sum_{\substack{n \in \Z^d \\ |n| \leq N}}
       \frac{1}{|\det A|^{j}}
       \sum_{\substack { D \in \mathcal{D}_A \\ \scale(D) \leq j}}
         |c_D| \Measure(D)^{-1/2} \Measure\bigl(D \cap A^j ([0,1]^d + \ell + n)\bigr).
\end{align*}
Note that $j = \scale(A^j([0,1]^d + \ell + n)) = \scale (D')$.
Therefore, taking the supremum over all $D' = A^j ([0,1]^d + \ell)$ for $j \in \mathbb{Z}$ and
$\ell \in \mathbb{Z}^d$ gives that
\begin{align*}
  \sup_{D' \in \mathcal{D}_A}
    T_{D'}
  & \leq \sum_{\substack{n \in \Z^d \\ |n| \leq N}}
           \sup_{j \in \mathbb{Z}, \ell \in \mathbb{Z}^d}
             \frac{1}{|\det A|^j}
             \sum_{\substack { D \in \mathcal{D}_A \\ \scale(D) \leq j}}
               |c_D| \Measure(D)^{-1/2} \Measure\bigl(D \cap A^j ([0,1]^d + \ell + n)\bigr) \\
  &\lesssim_{d,N} \sup_{D' \in \mathcal{D}_A}
                    \frac{1}{\Measure(D')}
                    \sum_{\substack { D \in \mathcal{D}_A \\ \scale(D) \leq \scale(D')}}
                      |c_D| m(D)^{-1/2} \Measure(D \cap D'),
\end{align*}
which completes the proof.
\end{proof}

For obtaining the actual dual characterization of the spaces $\dot{\mathbf{f}}^0_{1, \infty} (A)$
and $\dot{\mathbf{f}}^0_{\infty, 1}(A)$, the following lemma will be used.
It is \cite[Proposition 1.4]{hanninen2018equivalence} applied to
the special case of dilated cubes; see also \cite[Theorem 4]{verbitsky1996imbedding}
for the case of isotropic dilations.

\begin{lemma}[\cite{hanninen2018equivalence}]\label{lem:carleson_equiv}
  Let $a = (a_D)_{D \in \mathcal{D}_A}$ be a fixed but arbitrary
  sequence of non-negative reals.
  Then for every $C > 0$, the following assertions are equivalent:
  \begin{enumerate}[(i)]
  \item The sequence $a = (a_D)_{D \in \mathcal{D}_A}$ is a $C$-Carleson
        sequence, i.e.,
        \begin{align}
          \label{eq:carleson}
          \sum_{D \in \mathcal{D}'_A} a_D
          \leq C \Measure \bigg( \bigcup_{D \in \mathcal{D}'_A} D \bigg)
        \end{align}
        for every subcollection $\mathcal{D}'_A$ of the dilated cubes $\mathcal{D}_A$.

  \item For every sequence $b = (b_D)_{D \in \mathcal{D}_A}$ of
        non-negative reals, the estimate
        \[
          \sum_{D \in \mathcal{D}_A} a_D b_D \leq C \int_{\mathbb{R}^d}
          \sup_{D \in \mathcal{D}_A} b_D \mathds{1}_D (x) \; dx
        \]
        holds.
  \end{enumerate}
\end{lemma}

The significance of a Carleson sequence \eqref{eq:carleson} for the
purpose of the present paper is that it characterizes membership of
$\dot{\mathbf{f}}_{\infty, 1}^0 (A)$.
Although this fact is well-known for isotropic dilations
(cf.\ \cite{hanninen2018equivalence, verbitsky1996imbedding}),
the anisotropic version requires some additional arguments due to the fact
that dilated cubes are not necessarily nested.
The details are provided in the next lemma.

\begin{lemma}\label{lem:carleson_fi1}
  Let $A \in \mathrm{GL}(d, \mathbb{R})$ be expansive
  and let $(c_D)_{D \in \mathcal{D}_A}$ be a complex-valued sequence.
  Then $c \in \dot{\mathbf{f}}_{\infty, 1}^0(A)$ if, and only if, there exists $C > 0$ such that
  \begin{align} \label{eq:carleson_fi1}
   \sum_{D \in \mathcal{D}_A'} |c_D| \Measure(D)^{1/2}
   \leq C \Measure \bigg( \bigcup_{D \in \mathcal{D}'_A} D \bigg)
  \end{align}
  for every subcollection $\mathcal{D}'_A \subseteq \mathcal{D}_A$.
  Moreover,
  \begin{equation}\label{eq:fi1-equiv-pre}
    \| c \|_{\dot{\mathbf{f}}^0_{\infty, 1}(A)}
    \asymp \inf
           \bigg\{
             C >0
             :
             \sum_{D \in \mathcal{D}_A'}
               |c_D| \Measure(D)^{1/2}
             \leq C \Measure \bigg( \bigcup_{D \in \mathcal{D}'_A} D \bigg)
             \text{ for all } \mathcal{D}'_A \subseteq \mathcal{D}_A
           \bigg\},
  \end{equation}
  with implicit constant independent of $c$.
\end{lemma}

\begin{proof}
  First, it will be shown that if $(c_D)_{D \in \mathcal{D}_A}$ satisfies \eqref{eq:carleson_fi1},
  then $(c_D)_{D \in \mathcal{D}_A} \in \dot{\mathbf{f}}_{\infty, 1}^0(A)$.
  For this, let $D' \in \mathcal{D}_A$ be arbitrary.
  Then for any $C > 0$ satisfying \eqref{eq:carleson_fi1}, we have, by \Cref{lem:tent-superset},
  \begin{align*}
    \sum_{D \in \mathcal{T} (D')} |c_D| \Measure(D)^{1/2}
    \leq C \Measure \bigg(\bigcup_{D \in \mathcal{T}(D')} D \bigg)
    \leq C \Measure \bigg( \bigcup_{\substack{n \in \Z^d \\ |n| \leq N}} D' + A^{\scale(D')}n \bigg)
    \lesssim C \Measure(D'),
  \end{align*}
  with implicit constant independent of $D'$.
  Hence,
  \[
    \frac{1}{\Measure(D')}
    \sum_{D \in \mathcal{T}(D')} |c_D| \Measure(D)^{1/2}
    \lesssim C,
  \]
  which yields
  $\| c \|_{\dot{\mathbf{f}}_{\infty, 1}^0(A)} \lesssim C$ by \Cref{lem:fi1_char_tent}.
  This also implies $\lesssim$ in \Cref{eq:fi1-equiv-pre}.

  Conversely, let $\mathcal{D}_A' \subseteq \mathcal{D}_A$ be any subcollection.
  Note first that if, for all $N \in \mathbb{N}$, there exists
  some $D'\in \mathcal{D}_A'$ with $\scale(D') > N$, then
  \[
    \Measure \bigg( \bigcup_{D \in \mathcal{D}_A'} D \bigg)
    \geq \Measure(D')
    = |\det A|^{\scale(D')}
    > |\det A|^N
    .
  \]
  Hence, $\Measure \big( \bigcup_{D \in \mathcal{D}_A'} D \big) = \infty$ and
  \eqref{eq:carleson_fi1} is trivially satisfied.
  Therefore, suppose throughout the remainder of the proof that there exists
  $N \in \mathbb{N}$ such that $\scale(D) \leq N$ for all $D \in \mathcal{D}_A'$.
  Set $j_1 := \max \{ \scale(D) : D \in \mathcal{D}_A' \} \leq N$, and define
  \[
    \mathcal{D}_1''
    := \{ D \in \mathcal{D}_A' : \scale(D) = j_1 \}.
  \]
  Furthermore, set
  \(
    (\mathcal{D}_{1}'')^c
    := \{
         D \in \mathcal{D}_A'
         :
         D \notin \mathcal{T}(D') \; \text{for any} \; D' \in \mathcal{D}_1''
       \}
    .
  \)
  Observe that the elements of $\mathcal{D}_1''$ are pairwise disjoint up to measure zero.
  Moreover, by construction, the unions $\bigcup_{D' \in \mathcal{D}_1''} D'$ and
  $\bigcup_{D \in (\mathcal{D}_1'')^c} D$ are disjoint up measure zero and
  $ \mathcal{D}_A' \subseteq \bigcup_{D \in \mathcal{D}''_1} \mathcal{T} (D) \cup (\mathcal{D}_{1}'')^c$.

  For $\ell \geq 2$, we define $\mathcal{D}''_\ell$ inductively as follows:
  Set $j_\ell := \max \{ \scale(D) : D \in (\mathcal{D}_{\ell-1}'')^c \}$,
  \[
    \mathcal{D}_{\ell} ''
    := \{ D \in (\mathcal{D}_{\ell-1}'')^c : \scale(D) = j_\ell \},
  \]
  and
  \(
    (\mathcal{D}_{\ell}'')^c
    := \{
         D \in (\mathcal{D}_{\ell-1}'')^c
         :
         D \notin \mathcal{T}(D') \; \text{for any} \; D' \in \mathcal{D}_{\ell}''
       \}.
  \)
  Then, by construction, the dilated cubes in
  $\mathcal{D}_A'': = \bigcup_{\ell = 1}^{\infty} \mathcal{D}''_\ell$ are
  pairwise disjoint up to measure zero and
  \[
    \mathcal{D}_A'
    \subseteq \bigcup_{\ell = 1}^{\infty}
                \bigcup_{D \in \mathcal{D}''_\ell}
                  \mathcal{T} (D)
    = \bigcup_{D \in \mathcal{D}_A''} \mathcal{T} (D).
  \]

  Based on this construction, a direct calculation using \Cref{lem:fi1_char_tent} yields
  \begin{align*}
    \sum_{D \in \mathcal{D}_A'} |c_D| \Measure(D)^{1/2}
    &\leq \sum_{D' \in \mathcal{D}_A''}
      \sum_{D \in \mathcal{T}(D')} |c_D| \Measure(D)^{1/2}
    \lesssim \| c \|_{\dot{\mathbf{f}}^0_{\infty, 1}(A)}
      \sum_{D' \in \mathcal{D}_A''} \Measure(D') \\
    &= \| c \|_{\dot{\mathbf{f}}^0_{\infty, 1}(A)} \,
       \Measure \bigg( \bigcup_{D' \in \mathcal{D}_A''} D'\bigg)
    \leq \| c \|_{\dot{\mathbf{f}}^0_{\infty, 1}(A)} \,
         \Measure \bigg(\bigcup_{D \in \mathcal{D}_A'} D \bigg),
  \end{align*}
  where the last inequality used that $\mathcal{D}_A'' \subseteq \mathcal{D}_A'$.
  Hence $\| c \|_{\dot{\mathbf{f}}^0_{\infty, 1}(A)}$ satisfies
  \Cref{eq:carleson_fi1}, which also implies the inequality $\gtrsim$ in \eqref{eq:fi1-equiv-pre}.
\end{proof}

A combination of Lemmata \ref{lem:carleson_equiv} and \ref{lem:carleson_fi1}
yields the following dual characterization.

\begin{corollary}\label{lem:norming_sequence}
  Let $A \in \mathrm{GL}(d, \mathbb{R})$ be expansive.
  Then, for all $c \in \dot{\mathbf{f}}^0_{1, \infty} (A)$
  and $c' \in \dot{\mathbf{f}}^0_{\infty, 1} (A)$,
  \begin{align}\label{eq:pairing_seq_infty}
    | \langle c, c' \rangle |
    := \bigg| \sum_{D \in \mathcal{D}_A} c_{D} \overline {c'_{D}} \bigg|
    \lesssim \| c \|_{\dot{\mathbf{f}}^0_{1, \infty} (A)}  \| c' \|_{\dot{\mathbf{f}}^0_{\infty, 1} (A)} .
\end{align}
Moreover, it holds that
\begin{align} \label{eq:dual_sequence}
  \| c \|_{\dot{\mathbf{f}}^0_{1, \infty} (A)}
  \asymp \sup
         \big\{
           |\langle c, c' \rangle |
           \; :  \;
           c' \in \dot{\mathbf{f}}^0_{\infty, 1}(A), \;\;
           \| c' \|_{\dot{\mathbf{f}}^0_{ \infty,1}(A)} \leq 1
         \big\}.
\end{align}
\end{corollary}

\begin{proof}
  For $c' \in \dot{\mathbf{f}}^0_{ \infty, 1} (A)$ and
  $c \in \dot{\mathbf{f}}^0_{1, \infty} (A)$, define sequences by
  $a_D := |c'_D| \Measure(D)^{1/2}$ and $b_D := |c_D| \Measure(D)^{-1/2}$ for $D \in \mathcal{D}_A$.
  Then, by Lemma~\ref{lem:carleson_fi1}, we see that $(a_D)_{D \in \CalD_A}$ is a $C$-Carleson
  sequence, where $C \lesssim \| c' \|_{\dot{\mathbf{f}}^0_{\infty, 1} (A)}$.
  By Lemma~\ref{lem:carleson_equiv}, this implies
  \[
    \sum_{D \in \mathcal{D}_A} |c_D c'_D|
    = \sum_{D \in \mathcal{D}_A} a_D b_D \leq C \int_{\mathbb{R}^d}
    \sup_{D \in \mathcal{D}_A} b_D \mathds{1}_D (x) \; dx
    \lesssim \| c \|_{\dot{\mathbf{f}}^0_{1, \infty} (A)}
             \| c' \|_{\dot{\mathbf{f}}^0_{\infty, 1} (A)},
  \]
  showing \Cref{eq:pairing_seq_infty}.

  To obtain the dual characterization \eqref{eq:dual_sequence}, we follow
  \cite[Section~69]{zaanen1967integration} and define the associate
  norms of $\| \cdot \|_{\dot{\mathbf{f}}^0_{1, \infty}}$ by
  $\| \cdot \|^{(0)} := \| \cdot \|_{\dot{\mathbf{f}}^0_{1, \infty}}$ and
  \[
    \| c \|^{(n)}
    := \sup \bigg\{ \sum_{D \in \mathcal{D}'_A} |c_D c'_D| \; : \; \| c' \|^{(n-1)} \leq 1 \bigg\}
    = \sup \{ | \langle c, c' \rangle |  \; : \; \| c' \|^{(n-1)} \leq 1 \},
    \quad n \geq 1,
  \]
  where the equality can be shown using the solidity of the
  associate norms and choosing sequences $c'$ with appropriate (complex) signs;
  see also \cite[Section~69, Theorem~1]{zaanen1967integration} for details.
  In the following, we consider $\| \cdot \|^{(1)}$ and $\| \cdot \|^{(2)}$ in more detail.
  Starting with $\| \cdot \|^{(1)}$, we interpret the supremum as an infimum over
  all upper bounds.
  Then the characterizations of Lemma~\ref{lem:carleson_equiv} and \ref{lem:carleson_fi1} give
  \begin{align*}
    \| c \|^{(1)}
    &= \sup
       \bigg\{
         \sum_{D \in \mathcal{D}'_A} |c_D c'_D|
         :
         \| c' \|_{\dot{\mathbf{f}}^0_{1,\infty}}  \leq 1
       \bigg\} \\
    &= \inf
       \bigg\{
         C > 0
         :
         \sum_{D \in \mathcal{D}'_A} |c_D c'_D| \leq C \| c' \|_{\dot{\mathbf{f}}^0_{1,\infty}}
         \text{ for all } c' \in \dot{\mathbf{f}}^0_{1,\infty}(A)
       \bigg\} \\
    &= \inf
        \bigg\{
          C > 0
          :
          \sum_{D \in \mathcal{D}_A'}
            |c_D| \Measure(D)^{1/2}
          \leq C
               \Measure \bigg( \bigcup_{D \in \mathcal{D}'_A} D \bigg)
          \text{ for all } \mathcal{D}'_A \subseteq \mathcal{D}_A
        \bigg\} \\
    & \asymp \| c \|_{\dot{\mathbf{f}}^0_{\infty, 1}(A)}.
  \end{align*}
  The Lorentz-Luxemburg duality theorem for
  normed Köthe spaces (see, e.g., \cite[Section~71, Theorem~1]{zaanen1967integration}) states that
  $\| \cdot \|_{\dot{\mathbf{f}}^0_{1, \infty}} = \| \cdot \|^{(2)}$
  provided $ \| \cdot \|_{\dot{\mathbf{f}}^0_{1, \infty}}$ satisfies the Fatou property.
  Since the latter is a straightforward consequence of Fatou's lemma
  and \cite[Section~65, Theorem~3]{zaanen1967integration}), we obtain
  \[
    \| c \|_{\dot{\mathbf{f}}^0_{1, \infty} (A)}
    = \| c \|^{(2)}
    = \sup \big\{ |\langle c, c' \rangle | \; :  \; \| c' \|^{(1)} \leq 1 \big\}
    \asymp \sup \big\{ |\langle c, c' \rangle | \; :  \; \| c' \|_{\dot{\mathbf{f}}^0_{\infty, 1}(A)}
    \leq 1 \big\}
  \]
  for arbitrary $c \in \dot{\mathbf{f}}^0_{1, \infty} (A)$.
  This completes the proof.
\end{proof}

The final result of this section is the desired dual norm characterization
of $\dot{\mathbf{F}}^{0}_{1,\infty}(A)$.

\begin{proposition}\label{prop:dualnormF1inf}
  Let $A \in \mathrm{GL}(d, \mathbb{R})$ be expansive.
  Then, for all $g \in \Schwartz_0(\R^d)$,
  \[
    \| g \|_{\dot{\mathbf{F}}^{0}_{1,\infty}(A)}
    \asymp \sup_{\substack{ f \in \dot{\mathbf{F}}^0_{\infty, 1}  \\
    \| f \|_{\dot{\mathbf{F}}^0_{\infty, 1}} \leq 1} } |\langle f,
    g\rangle|.
  \]
\end{proposition}

\begin{proof}
  By \cite[Theorem 3.12]{bownik2007anisotropic}, there exists a function
  $\psi \in \mathcal{S} (\mathbb{R}^d)$ with compact Fourier support
  such that the operator
  $\mathscr{C}_{\psi} f = (\langle f, \psi_D \rangle)_{D \in \mathcal{D}}$
  is bounded from $\TLzero(A)$ into $\dot{\mathbf{f}}^{0}_{p,q}(A)$ and furthermore
  the operator $\mathscr{D}_{\psi} c = \sum_{D \in \mathcal{D}} c_D \psi_D$ is
  bounded from $\dot{\mathbf{f}}^{0}_{p,q}(A)$ into $\TLzero(A)$ for all $p,q \in (0, \infty]$.
  Moreover, their composition $\mathscr{D}_{\psi} \circ \mathscr{C}_{\psi}$ is the identity on
  $\TLzero(A)$.
  Here, for $D = A^{j}([0,1]^d + k)$, the function $\psi_D$ is defined as
  $\psi_D (x) = |\det A|^{-j/2} \psi(A^{-j} x - k)$;
  see \cite[Equation~(2.9)]{bownik2007anisotropic}.

  Next, \cite[Lemma 2.8]{bownik2006atomic} implies for all
  $f \!\in\! {\dot{\mathbf{F}}^0_{\infty, 1} (A)} \!\subseteq \mathcal{S}' / \mathcal{P} \cong \Schwartz_0'(\R^d)$
  and $g \!\in\! \mathcal{S}_0 (\mathbb{R}^d) \subseteq {\dot{\mathbf{F}}^0_{1, \infty} (A)}$ that
  \[
    \langle f, g \rangle
    = \sum_{D \in \mathcal{D}_A}
        \langle f, \psi_{D} \rangle
        \langle \psi_{D} , g \rangle
    = \sum_{D \in \mathcal{D}_A}
        \langle f, \psi_{D} \rangle
        \overline{\langle g, \psi_{D} \rangle}
    = \big \langle \mathscr{C}_{\psi} f , \mathscr{C}_{\psi} g \rangle .
  \]
  Combining both facts with the estimate \eqref{eq:pairing_seq_infty}, it follows that
  \[
    |\langle f, g \rangle | = | \langle \mathscr{C}_{\psi} f, \mathscr{C}_{\psi} g \rangle |
    \lesssim \| \mathscr{C}_{\psi} f \|_{\dot{\mathbf{f}}^0_{\infty, 1} (A)}
    \| \mathscr{C}_{\psi} g \|_{\dot{\mathbf{f}}^0_{1, \infty} (A)}
    \lesssim \| f \|_{\dot{\mathbf{F}}^0_{\infty, 1} (A)} \| g \|_{\dot{\mathbf{F}}^0_{1, \infty} (A)} .
  \]

  For the reverse inequality, first note that since $\mathscr{D}_{\psi} \circ \mathscr{C}_{\psi}$
  is the identity on $\dot{\mathbf{F}}^0_{1, \infty} (A)$, and since these operators are bounded,
  we have
  \[
    \| g \|_{\dot{\mathbf{F}}^0_{1, \infty} (A)}
    = \| \mathscr{D}_{\psi} \mathscr{C}_{\psi} g \|_{\dot{\mathbf{F}}^0_{1, \infty} (A)}
    \lesssim \| \mathscr{C}_{\psi} g \|_{\dot{\mathbf{f}}^0_{1, \infty} (A)}
    \lesssim \| g \|_{\dot{\mathbf{F}}^0_{1, \infty} (A)}
  \]
  and thus
  \(
    \| \mathscr{C}_{\psi} g \|_{\dot{\mathbf{f}}^0_{1, \infty} (A)}
    \asymp \| g \|_{\dot{\mathbf{F}}^0_{1, \infty} (A)}.
  \)
  Next, note that by \Cref{lem:norming_sequence}, there exists a sequence
  $(c^{(n)})_{n \in \N}$ in $\dot{\mathbf{f}}^0_{ \infty, 1} (A)$ with
  $\| c^{(n)} \|_{\dot{\mathbf{f}}^0_{ \infty, 1} (A)} \leq 1$ such that
  \[
    \limsup_{n \to \infty}| \langle \mathscr{C}_{\psi} g,  c^{(n)} \rangle | \asymp \|
    \mathscr{C}_{\psi} g \|_{\dot{\mathbf{f}}^0_{1, \infty} (A)}.
  \]

  Now, setting $f_n := \mathscr{D}_{\psi} c^{(n)} \in \dot{\mathbf{F}}^0_{\infty, 1} (A)$, note that
  \[
    \langle f_n, g \rangle
    = \langle \mathscr{D}_{\psi} c^{(n)}, g \rangle
    = \sum_{D \in \CalD_A}
        c_D^{(n)} \, \langle \psi_D, g \rangle
    = \sum_{D \in \CalD_A}
        c_D^{(n)} \, \overline{(\mathscr{C}_{\psi} g)_D}
    = \langle c^{(n)}, \mathscr{C}_{\psi} g \rangle
    .
  \]
  It follows that
  \[
    \limsup_{n \to \infty}
      | \langle f_n, g \rangle |
    = \limsup_{n \to \infty}
        |\langle c^{(n)}, \mathscr{C}_{\psi} g \rangle |
    = \limsup_{n \to \infty}
        | \langle \mathscr{C}_{\psi} g,  c^{(n)} \rangle |
    \asymp \| \mathscr{C}_{\psi} g \|_{\dot{\mathbf{f}}^0_{1, \infty} (A)}
    \asymp \| g \|_{\dot{\mathbf{F}}^0_{1, \infty} (A)}
    .
  \]
  Since $\mathscr{D}_{\psi} \colon \dot{\mathbf{f}}^0_{\infty, 1} (A) \to \dot{\mathbf{F}}^0_{\infty,1} (A)$
  is bounded, we have
  \(
    \|f_n\|_{\dot{\mathbf{F}}^0_{\infty,1} (A)}
    \leq C \|c^{(n)}\|_{\dot{\mathbf{f}}^0_{\infty,1} (A)}
    \leq C
  \)
  for all $n \in \N$.
  Hence, normalizing the $f_n$ by $C > 0$ if necessary yields that
  \[
    \| g \|_{\dot{\mathbf{F}}^0_{1, \infty} (A)}
    \lesssim
    \sup_{\substack{ f \in \dot{\mathbf{F}}^0_{\infty, 1} \\
        \| f \|_{\dot{\mathbf{F}}^0_{\infty, 1}} \leq 1 }} |\langle f,
    g\rangle|,
  \]
  which finishes the proof.
\end{proof}

\section*{Acknowledgements}
J.v.V. gratefully acknowledges support from the Austrian Science Fund (FWF) project J-4445 and is grateful for the hospitality and
support of the Katholische Universität Eichstätt-Ingolstadt during his visit. J.v.V. also thanks Emiel Lorist for helpful discussions on discrete Triebel-Lizorkin spaces.

\end{document}